\documentclass[11pt,a4paper]{amsart}
\title[Gromov--Witten Invariants of Blow-Ups] {The Abelian/non-Abelian Correspondence and Gromov--Witten Invariants of Blow-Ups}

\author[Coates]{Tom Coates}
\address{Department of Mathematics\\
Imperial College London\\
180 Queen's Gate\\
London SW7 2AZ
\\UK}
\email{t.coates@imperial.ac.uk}

\author[Lutz]{Wendelin Lutz}
\address{Department of Mathematics\\
Imperial College London\\
180 Queen's Gate\\
London SW7 2AZ
\\UK}
\email{wl4714@imperial.ac.uk}

\author[Shafi]{Qaasim Shafi}
\address{Department of Mathematics\\
Imperial College London\\
180 Queen's Gate\\
London SW7 2AZ\\
UK}
\email{mqs14@imperial.ac.uk}

\usepackage{amsfonts}
\usepackage{tikz}
\usepackage{tikz-cd}
\usepackage{amssymb}
\usepackage[margin=2.5cm]{geometry}
\usepackage{booktabs}
\usepackage{hyperref}
\usepackage{xfrac}
\usepackage{epsfig}
\usepackage{enumitem}
\usepackage{tcolorbox}
\usepackage{mathtools} 

\hypersetup{
	colorlinks=true, 
	linkcolor=black, 
	citecolor=black
}

\DeclareMathOperator{\Ann}{{Ann}}
\DeclareMathOperator{\Bl}{{Bl}}
\DeclareMathOperator{\ch}{ch}
\newcommand{\Cstar}{\CC^\times}

\DeclareMathOperator{\ev}{{ev}}

\DeclareMathOperator{\Fl}{{Fl}}

\DeclareMathOperator{\GL}{{GL}}
\DeclareMathOperator{\GM}{{GM}}
\DeclareMathOperator{\Gr}{{Gr}}
\DeclareMathOperator{\Hom}{{Hom}}

\DeclareMathOperator{\image}{{im}}

\DeclareMathOperator{\NE}{{NE}}

\DeclareMathOperator{\Pic}{{Pic}}

\DeclareMathOperator{\rk}{{rk}}
\DeclareMathOperator{\Spec}{{Spec}}

\newcommand{\tw}{\text{\rm tw}}

\DeclareMathOperator{\cH}{\mathcal{H}}

\DeclareMathOperator{\cL}{\mathcal{L}}
\DeclareMathOperator{\cO}{\mathcal{O}}

\DeclareMathOperator{\CC}{\mathbb{C}}

\DeclareMathOperator{\OO}{\mathcal{O}}
\DeclareMathOperator{\PP}{\mathbb{P}}
\DeclareMathOperator{\QQ}{\mathbb{Q}}
\DeclareMathOperator{\RR}{\mathbb{R}}

\DeclareMathOperator{\ZZ}{\mathbb{Z}}

\newcommand{\bt}{{\bf t}}

\newcommand{\GIT}{/\!\!/}

\newcommand{\ip}{\raise1pt\hbox{\large $\lrcorner$}}

\theoremstyle{plain}
\newtheorem{theorem}{Theorem}[section]
\newtheorem{proposition}[theorem]{Proposition}
\newtheorem{lemma}[theorem]{Lemma}
\newtheorem*{lem*}{Lemma}

\newtheorem{corollary}[theorem]{Corollary}
\newtheorem{conjecture}[theorem]{Conjecture}

\theoremstyle{definition}
\newtheorem{definition}[theorem]{Definition}

\newtheorem{remark}[theorem]{Remark}
\newtheorem*{remark*}{Remark}
\newtheorem{example}[theorem]{Example}
\newtheorem{assumption}[theorem]{Assumption}

\begin{document}

\begin{abstract}
	We prove the Abelian/non-Abelian Correspondence with bundles for target spaces that are partial flag bundles, combining and generalising results by Ciocan-Fontanine--Kim--Sabbah, Brown, and Oh. From this we deduce how genus-zero Gromov--Witten invariants change when a smooth projective variety $X$ is blown up in a complete intersection defined by convex line bundles. In the case where the blow-up is Fano, our result gives closed-form expressions for certain genus-zero invariants of the blow-up in terms of invariants of $X$. We also give a reformulation of the Abelian/non-Abelian Correspondence in terms of Givental's formalism, which may be of independent interest. 
\end{abstract}
\maketitle
\section{Introduction}

Gromov--Witten invariants, roughly speaking, count the number of curves in a projective variety $X$ that are constrained to pass through various cycles. They play an essential role in mirror symmetry, and have been the focus of intense activity in symplectic and algebraic geometry over the last 25 years. Despite this, there are few effective tools for computing the Gromov--Witten invariants of blow-ups. In this paper we improve the situation somewhat: we determine how genus-zero Gromov--Witten invariants change when a smooth projective variety $X$ is blown up in a complete intersection of convex line bundles. In the case where the blow-up $\tilde{X}$ is Fano, a special case of our result gives closed-form expressions for genus-zero one-point descendant invariants of $\tilde{X}$ in terms of invariants of $X$, and hence determines the small $J$-function of $\tilde{X}$.

Suppose that $Z \subset X$ is the zero locus of a regular section of a direct sum of convex (or nef) line bundles $$E = L_0 \oplus \cdots \oplus L_r \to X$$ and that $\tilde{X}$ is the blow-up of $X$ in $Z$. To determine the genus-zero Gromov--Witten invariants of $\tilde{X}$, we proceed in two steps. First, we exhibit $\tilde{X}$ as the zero locus of a section of a convex vector bundle on the bundle of Grassmannians $\Gr(r,E^\vee) \to X$: this is Theorem~\ref{step one} below. We then establish a version of the Abelian/non-Abelian Correspondence~\cite{CFKS2008} that determines the genus-zero Gromov--Witten invariants of such zero loci. This is the Abelian/non-Abelian Correspondence with bundles, for target spaces that are partial flag bundles -- see Theorem~\ref{step two}. It builds on and generalises results by Ciocan-Fontanine--Kim--Sabbah~\cite[\S6]{CFKS2008}, Brown~\cite{Brown2014}, and Oh~\cite{Oh2016}.

\begin{theorem}[see Proposition~\ref{geometricconstruction} below for a more general result] \label{step one}
	Let $X$ be a smooth projective variety, let $E = L_0 \oplus \cdots \oplus L_r \to X$ be a direct sum of line bundles, and let $Z \subset X$ be the zero locus of a regular section $s$ of $E$. Let $\pi \colon \Gr(r,E^\vee) \to X$ be the Grassmann bundle of subspaces and let $S \to \Gr(r,E^\vee)$ be the tautological subbundle. Then the composition $$S \hookrightarrow \pi^*E^\vee \xrightarrow{\pi^* s^\vee} \cO $$ defines a regular section of $S^\vee$, and the zero locus of this section is the blow-up $\tilde{X} = \Bl_Z X$.
\end{theorem}

\noindent If the line bundles $L_i$ are convex, then the bundle $S^\vee$ is also convex.
The fact that $\tilde{X}$ is regularly embedded into $\Gr(r, E^\vee)\cong \PP(E)$ (where $\PP(E)$ is the projective bundle of lines) is well-known and true in more generality, see for example \cite[Appendix B8.2]{FultonIntersection} and \cite[Lemma $2.1$]{AluffiChernClasses}. However, to apply the Abelian/non-Abelian correpondence, the crucial point is that $\tilde{X}$ is cut out by a regular section of an explicit representation-theoretic bundle on $\Gr(r, E^\vee)$. Although this should be well-known to experts, we have been unable to find a reference for this.\\
To apply Theorem~\ref{step one} to Gromov--Witten theory, and to state the Abelian/non-Abelian Correspondence, we will use Givental's formalism~\cite{Givental2004}. This is a language for working with Gromov--Witten invariants and operations on them, in terms of linear symplectic geometry. We give details in \S\ref{givental formalism} below, but the key ingredients are, for each smooth projective variety $Y$, an infinite-dimensional symplectic vector space $\cH_Y$ called the Givental space and a Lagrangian submanifold $\cL_Y \subset \cH_Y$. Genus-zero Gromov--Witten invariants of $Y$ determine and are determined by $\cL_Y$. 

We will also consider \emph{twisted} Gromov--Witten invariants~\cite{CoatesGivental2007}. These are invariants of a projective variety $Y$ which depend also on a bundle $F \to Y$ and a characteristic class $\mathbf{c}$. For us, this characteristic class will always be the equivariant Euler class (or total Chern class)\begin{align} \label{intro equivariant Euler}
	\mathbf{c}(V) = \sum_{k=0}^d \lambda^{d-k} c_k(V) && \text{where $d$ is the rank of the vector bundle $V$.}
\end{align}
The parameter $\lambda$ here can be thought of as the generator for the $S^1$-equivariant cohomology of a point. There is a Lagrangian submanifold $\cL_{F_\lambda} \subset \cH_Y$ that encodes genus-zero Euler-twisted invariants of $Y$;  the Quantum Riemann--Roch theorem~\cite{CoatesGivental2007} implies that
\[
	\Delta_{F_\lambda} \cL_Y = \cL_{F_\lambda}
\]
where $\Delta_{F_\lambda} \colon \cH_Y \to \cH_Y$ is a certain linear symplectomorphism. This gives a family of Lagrangian submanifolds $\lambda \mapsto \cL_{F_\lambda}$ defined over $\QQ(\lambda)$, that is, a meromorphic family of Lagrangian submanifolds parameterised by $\lambda$.  When $F$ satisfies a positivity condition called convexity, the family $\lambda \mapsto \cL_\lambda$ extends analytically across $\lambda=0$ and the limit $\cL_{F_0}$ exists. This limiting submanifold $\cL_{F_0} \subset \cH_Y$ determines genus-zero Gromov--Witten invariants of the subvariety of $Y$ cut out by a generic section of $F$~\cite{CoatesGivental2007,Coates2014}. Theorem~\ref{step one} therefore allows us to determine genus-zero Gromov--Witten invariants of the blow-up $\tilde{X}$, by analyzing the limiting submanifold~$\cL_{S^\vee_0}$.

Our second main result, Theorem~\ref{step two}, applies to the Grassmann bundle $\Gr(r,E^\vee) \to X$ considered in Theorem~\ref{step one}, and more generally to any partial flag bundle $\Fl(E) \to X$ induced by $E$. Such a partial flag bundle can be expressed as a GIT quotient $A \GIT G$, where $G$ is a product of general linear groups, and so any representation $\rho$ of $G$ on a vector space $V$ induces a vector bundle $V^G \to \Fl(E)$ with fiber $V$. See \S\ref{flag} for details of the construction. We give an explicit family of elements of $\cH_{\Fl(E)}$, 
\begin{align} \label{intro twisted I}
	(t, \tau) \mapsto I_{\GM}(t, \tau, z) && \text{$t \in \CC^R$ for some $R$, $\tau \in H^\bullet(X)$}
\end{align}
defined in terms of genus-zero Gromov--Witten invariants of $X$ and explicit hypergeometric functions, and show that this family, after changing the sign of $z$, lies on the Lagrangian submanifold that determines Euler-twisted Gromov--Witten invariants of $\Fl(E)$ with respect to~$V^G$.

\begin{theorem}[see Definition~\ref{IGM definition special case} and Theorem~\ref{IGM on twisted cone}]\label{step two}
	For all $t \in \CC^R$ and $\tau \in H^\bullet(X)$,
	$$I_{\GM}(t, \tau, {-z}) \in \cL_{V^G_\lambda}$$
\end{theorem}

\noindent Under an ampleness condition -- which holds, for example, whenever the blow-up $\tilde{X}$ in Theorem~\ref{step one} is Fano -- the family \eqref{intro twisted I} takes a particularly simple form
$$ 
I_{\GM}(t, \tau, z) = z \left(1 + o(z^{-1})\right)
$$
and standard techniques in Givental formalism allow us to determine genus-zero twisted Gromov--Witten invariants of $\Fl(E)$ explicitly: see Corollaries~\ref{I=J} and~\ref{explicit Gr}. Applying this in the setting of Theorem~\ref{step one}, we recover genus-zero Gromov--Witten invariants of the blow-up $\tilde{X}$ by taking the non-equivariant limit $\lambda \to 0$.

The reader who is focussed on blow-ups can stop reading here, jumping to the end of the Introduction for connections to previous work, \S\ref{flag} for basic setup, Corollary~\ref{explicit Gr} for the key Gromov--Witten theoretic result, and then to \S\ref{examples} for worked examples. In the rest of the Introduction, we explain how Theorem~\ref{step two} should be regarded as an instance of the Abelian/non-Abelian Correspondence~\cite{CFKS2008}.

\pagebreak

The Abelian/non-Abelian Correspondence relates the genus-zero Gromov--Witten theory of quotients $A \GIT G$ and $A \GIT T$, where $A$ is a smooth quasiprojective variety equipped with the action of a reductive Lie group $G$, and $T$ is its maximal torus. We fix a linearisation of this action such that the stable and semistable loci coincide and we suppose that the quotients $A \GIT G$ and $A \GIT T$ are smooth. In our setting the non-Abelian quotient $A \GIT G$ will be a partial flag bundle or Grassmann bundle over $X$, and the Abelian quotient $A \GIT T$ will be a bundle of toric varieties over $X$, that is, a toric bundle in the sense of Brown~\cite{Brown2014}. To reformulate the Abelian/non-Abelian Correspondence of~\cite{CFKS2008} in terms of Givental's formalism, however, we pass to the following more general situation. Let $W$ denote the Weyl group of $T$ in $G$. A theorem of Martin (Theorem~\ref{thm:Martin} below) expresses the cohomology of the non-Abelian quotient $H^\bullet(A \GIT G)$ as a quotient of the Weyl-invariant part of the cohomology of the Abelian quotient $H^\bullet(A \GIT T)^W$ by an appropriate ideal, so there is a quotient map 
\begin{equation} \label{quotient}
	H^\bullet(A \GIT T)^W \to H^\bullet(A \GIT G).
\end{equation}
The Abelian/non-Abelian Correspondence, in the form that we state it below, asserts that this map also controls the relationship between the quantum cohomology of $A \GIT G$ and $A \GIT T$. 

When comparing the quantum cohomology algebras of $A \GIT G$ and $A \GIT T$, or when comparing the Givental spaces of $A \GIT G$ and $A \GIT T$, we need to account for the fact that there are fewer curve classes on $A \GIT G$ than there are on $A \GIT T$. We do this as follows. The Givental space $\cH_Y$ discussed above is defined using cohomology groups $H^\bullet(Y;\Lambda)$ where $\Lambda$ is the Novikov ring for~$Y$: see \S\ref{givental formalism}. The Novikov ring contains formal linear combinations of terms $Q^d$ where $d$ is a curve class on~$Y$. The quotient map \eqref{quotient} induces an isomorphism $H^2(A \GIT T)^W \cong H^2(A \GIT G)$, and by duality this gives a map $\varrho \colon \NE(A \GIT T) \rightarrow \NE (A \GIT G)$ where $\NE$ denotes the Mori cone: see Proposition~\ref{maponmori}. Combining the quotient map \eqref{quotient} with the map on Novikov rings induced by $\varrho$ gives a map
\begin{equation} \label{quotientH}
	   p \colon \cH^W_{A \GIT T} \to \cH_{A \GIT G}
\end{equation}
between the Weyl-invariant part of the Givental space for the Abelian quotient and the Givental space for the non-Abelian quotient.  Here, and also below when we discuss Weyl-invariant functions, we consider the Weyl group $W$ to act on $\cH_{A \GIT T}$ through the combination of its action on cohomology classes and its action on the Novikov ring.

We consider now an appropriate twisted Gromov--Witten theory of $A \GIT T$.
For each root $\rho$ of~$G$, write $L_\rho \to A \GIT T$ for the line bundle determined by $\rho$, and let $\Phi = \oplus_\rho L_\rho$ where the sum runs over all roots. Consider the Lagrangian submanifold~$\cL_{\Phi_\lambda}$ that encodes genus-zero twisted Gromov--Witten invariants of $A \GIT T$. The bundle $\Phi$ is very far from convex, so one cannot expect the non-equivariant limit of $\cL_{\Phi_\lambda}$ to exist. Nonetheless, the projection along \eqref{quotientH} of the Weyl-invariant part of this Lagrangian submanifold does have a non-equivariant limit.

\begin{theorem}(see Corollary~\ref{GMlimit no bundle})
	The limit as $\lambda \to 0$ of $p \Big( \cL_{\Phi_\lambda} \cap \cH^W_{A \GIT T}\Big)$ exists.
\end{theorem}

\noindent We call this non-equivariant limit the \emph{Givental--Martin cone\footnote{We have not emphasised this point, but the Lagrangian submanifolds $\cL_Y$, $\cL_{F_{\lambda}}$, etc.~are in fact cones~\cite{Givental2004}.}} $\cL_{\GM} \subset \cH_{A\GIT G}$.

\begin{conjecture}[The Abelian/non-Abelian Correspondence] \label{AnA}
	$\cL_{\GM} = \cL_{A \GIT G}$.
\end{conjecture}

\noindent This is a reformulation of \cite[Conjecture~3.7.1]{CFKS2008}. The analogous statement for twisted Gromov--Witten invariants is the Abelian/non-Abelian Correspondence with bundles; this is a reformulation of  \cite[Conjecture~6.1.1]{CFKS2008}. Fix a representation $\rho$ of $G$, and consider the vector bundles $V^G \to A \GIT G$ and $V^T \to A \GIT T$ induced by~$\rho$. Consider the Lagrangian submanifold $\cL_{\Phi_{\lambda} \oplus V^T_{\mu}}$ that encodes genus-zero twisted
Gromov–Witten invariants of $A \GIT T$, where for the twist by the root bundle $\Phi$ we use
the equivariant Euler class \eqref{intro equivariant Euler} with parameter $\lambda$ and for the twist by $V^T$ we use the equivariant Euler class with a different parameter $\mu$. As before, the projection along \eqref{quotientH} of the Weyl-invariant part of this Lagrangian submanifold has a non-equivariant limit with respect to $\lambda$.
\begin{theorem}(see Theorem~\ref{GMlimit})
	The limit as $\lambda \to 0$ of $p \Big(\cL_{\Phi_\lambda \oplus V^T_\mu} \cap \cH^W_{A \GIT T}\Big)$ exists.
\end{theorem}

\noindent Let us call this limit the \emph{twisted Givental--Martin cone} $\cL_{\GM,V^T_\mu} \subset \cH_{A \GIT G}$. 

\begin{conjecture}[The Abelian/non-Abelian Correspondence with bundles]\label{AnA bundles}
	$\cL_{\GM, V^T_\mu} = \cL_{V^G_\mu}$.
\end{conjecture}

As in~\cite{CFKS2008}, the Abelian/non-Abelian Correspondence implies the Abelian/non-Abelian Correspondence with bundles.

\begin{proposition}
	Conjectures~\ref{AnA} and~\ref{AnA bundles} are equivalent.
\end{proposition}

\begin{proof}
	Conjecture~\ref{AnA} is the special case of Conjecture~\ref{AnA bundles} where the vector bundles involved have rank zero. To see that Conjecture~\ref{AnA} implies Conjecture~\ref{AnA bundles}, observe that the projection of the Quantum Riemann--Roch operator $\Delta_{V^T_\mu}$ under the map \eqref{quotientH} is $\Delta_{V^G_\mu}$: see Definition~\ref{delta}. Now apply the Quantum Riemann--Roch theorem~\cite{CoatesGivental2007}.
\end{proof}

The following reformulations will also be useful. Given any Weyl-invariant family 
\begin{align*}
	t \mapsto I(t) \in \cH^W_{A \GIT T}
	&&  \text{of the form} &&
	I(t) = \sum_{d \in \NE(A \GIT T)} Q^d I_d(t)
\end{align*}
we define its \emph{Weyl modification} $t \mapsto \widetilde{I}(t) \in  \cH^W_{A \GIT T}$ to be
$$ \widetilde{I}(t) = \sum_{d \in \NE(A \GIT T)} Q^d W_d I_d(t) $$
where $W_d$ is an explicit hypergeometric factor that depends on $\lambda$ -- see~\eqref{modg}.  We prove in Lemma~\ref{IGMexists} below that, for a Weyl-invariant family $t \mapsto I(t)$ the image under \eqref{quotientH} of the Weyl modification $t \mapsto p(\widetilde{I}(t))$ has a well-defined limit as $\lambda \to 0$. We call this limit the \emph{Givental--Martin modification} of $t \mapsto I(t)$ and denote it by $t \mapsto I_{\GM}(t)$; it is a family of elements of $\cH_{A \GIT G}$. Furthermore, if $t \mapsto I(t)$  satisfies the Divisor Equation in the sense of equation~\eqref{divisor equation}, then:
\begin{itemize}[itemsep=0.5ex]	
	\item if $t \mapsto I(t)$ is a family of elements of $\cL_{A \GIT T}$ then $t \mapsto I_{\GM}(t)$ is a family of elements on the Givental--Martin cone $\cL_{\GM}$; and 
	\item if $t \mapsto I(t)$ is a family of elements of the twisted cone $\cL_{V^T_\mu}$ then $t \mapsto I_{\GM}(t)$ is a family of elements on the twisted Givental--Martin cone $\cL_{\GM,V^T_\mu}$.
\end{itemize}	
The first statement here is Corollary~\ref{IGMonLGM} with $F'=0$; the second statement is Corollary~\ref{IGMonLGM}. This lets us reformulate the Abelian/non-Abelian Correspondence in more concrete terms.

\begin{conjecture}[a reformulation of Conjecture~\ref{AnA}] \label{AnA family}
	Let $t \mapsto I(t)$ be a Weyl-invariant family of elements of $\cL_{A \GIT T}$ that satisfies the Divisor Equation. Then the Givental--Martin modification $t \mapsto I_{\GM}(t)$ is a family of elements of $\cL_{A \GIT G}$. 
\end{conjecture}

\begin{conjecture}[a reformulation of Conjecture~\ref{AnA bundles}] \label{AnA bundles family}
	Let $t \mapsto I(t)$ be a Weyl-invariant family of elements of $\cL_{V^T_\mu}$ that satisfies the Divisor Equation. Then the Givental--Martin modification $t \mapsto I_{\GM}(t)$ is a family of elements of $\cL_{V^G_\mu}$. 
\end{conjecture}

Let us now specialise to the case of partial flag bundles, as in \S\ref{notation} and the rest of the paper, so that $A \GIT G$ is a partial flag bundle $\Fl(E) \to X$ and $A \GIT T$ is a toric bundle $\Fl(E)_T \to X$. Theorem~\ref{brownoh AnA} below establishes the statement of Conjecture~\ref{AnA family} not for an arbitrary Weyl-invariant family $t \mapsto I(t)$ on $\cL_{A \GIT T}$, but for a specific such family called the \emph{Brown $I$-function}. As we recall in Theorems~\ref{ohI} and~\ref{brown2014gromov}, Brown and Oh have defined families $t \mapsto I_{\Fl(E)_T}(t)$ and $t \mapsto I_{\Fl(E)}(t)$, given in terms of genus-zero Gromov--Witten invariants of $X$ and explicit hypergeometric functions, and have shown~\cite{Brown2014, Oh2016} that $I_{\Fl(E)_T}(t) \in \cL_{\Fl(E)_T}$ and $I_{\Fl(E)}(t) \in \cL_{\Fl(E)}$. 

\begin{theorem}[see Proposition~\ref{brownoh} for details] \label{brintroh} \label{brownoh AnA}
	The Givental--Martin modification of the Brown $I$-function $t \mapsto I_{\Fl(E)_T}$ is $t \mapsto I_{\Fl(E)}(t)$.
\end{theorem}

\noindent The main result of this paper is the analogue of Theorem~\ref{brownoh AnA} for twisted Gromov--Witten invariants. We define a twisted version $t \mapsto I_{V^T_\mu}(t)$ of the Brown $I$-function and prove:

\begin{theorem}[see Definition~\ref{IGM definition special case} and Corollary~\ref{IGM on twisted cone} for details] \ \label{step three}
	\begin{enumerate}
		\item the twisted Brown $I$-function $t \mapsto I_{V^T_\mu}(t)$ is a Weyl-invariant family of elements of~$\cL_{V^T_\mu}$ that satisfies the Divisor Equation;
		\item the Givental--Martin modification $t \mapsto I_{\GM}(t)$ of this family satisfies $I_{\GM}(t) \in \cL_{V^G_\mu}$.
	\end{enumerate}
\end{theorem}

\noindent This establishes the statement of Conjecture~\ref{AnA bundles family}, not for an arbitrary Weyl-invariant family, but for the specific such family $t \mapsto I_{V^T_\mu}(t)$. Theorem~\ref{step three} follows from the Quantum Riemann--Roch theorem~\cite{CoatesGivental2007} together with the results of Brown~\cite{Brown2014} and Oh~\cite{Oh2016}, using a ``twisting the $I$-function'' argument as in~\cite{CCIT2019}.

As we will now explain, Theorem~\ref{brownoh AnA} is quite close to a proof of Conjecture~\ref{AnA family} in the flag bundle case, and similarly Theorem~\ref{step three} is close to a proof of Conjecture~\ref{AnA bundles family}. We will discuss only the former, as the latter is very similar. Theorem~\ref{brownoh AnA} implies that
\begin{equation}
	\label{AnA intermediate}
	\text{the Givental--Martin modification $t \mapsto I_{\GM}(t)$ lies in $\cL_{\Fl(E)}$}
\end{equation}
for the family $t \mapsto I(t)$ given by the Brown I-function, because the Givental--Martin modification of the Brown $I$-function is the Oh $I$-function $t \mapsto I_{\Fl(E)}(t)$. If Oh's $I$-function were a \emph{big $I$-function}, in the sense of~\cite{CFK2016}, then Conjecture~\ref{AnA family} would follow. The special geometric properties of the Lagrangian submanifold $\cL_Y$ described in~\cite{Givental2004} and~\cite[Appendix B]{CCIT2009Computing}, taking $Y = \Fl(E)$, would then imply that any family $t \mapsto I(t)$ such that $I(t) \in \cL_{\Fl(E)}$ can be written as
\begin{equation}
	\label{special form}
	I(t) = I_{\Fl(E)}(\tau(t)) + \sum_\alpha C_\alpha(t, z) z \frac{\partial I_{\Fl(E)}}{\partial \tau_\alpha}(\tau(t))
\end{equation}
for some coefficients $C_\alpha(t, z)$ that depend polynomially on $z$ and some change of variables $t \mapsto \tau(t)$. Furthermore the same geometric properties imply that any family of the form \eqref{special form} satisfies $I(t) \in \cL_{\Fl(E)}$. But $\cL_{\GM}$ has the same special geometric properties as $\cL_Y$ -- it inherits them from the Weyl-invariant part of $\cL_{\Phi_\lambda}$ by projection along \eqref{quotientH} followed by taking the non-equivariant limit -- and so if $t \mapsto I_{\Fl(E)}$ is a big $I$-function then any family of elements $t \mapsto I^\dagger(t)$ on $\cL_{\GM}$ can be written as 
\begin{equation*}
	I^\dagger(t) = I_{\Fl(E)}(\tau^\dagger(t)) + \sum_\alpha C^\dagger_\alpha(t, z) z \frac{\partial I_{\Fl(E)}}{\partial \tau_\alpha}(\tau^\dagger(t))
\end{equation*}
That is, $I^\dagger(t)$ can be written in the form \eqref{special form}. It follows that $I^\dagger(t) \in \cL_{\Fl(E)}$.  Applying this with $I^\dagger = I_{\GM}$ from Conjecture~\ref{AnA family} proves that Conjecture; note that we know that the family $t \mapsto I_{\GM}(t)$ here lies in $\cL_{\GM}$ by Corollary~\ref{IGMonLGM}.

If the Brown and Oh $I$-functions were big $I$-functions then Theorem~\ref{brownoh AnA} would continue to hold (with the same proof) and Conjecture~\ref{AnA family} would therefore follow. In reality the Brown and Oh $I$-functions are only small $I$-functions, not big $I$-functions, but Ciocan-Fontanine--Kim have explained in \cite[\S5]{CFK2016} how to pass from small $I$-functions to big $I$-functions, whenever the target space is the GIT quotient of a vector space. To apply their argument, and hence prove Conjecture~\ref{AnA family} for partial flag bundles, one would need to check that the Brown $I$-function arises from torus localization on an appropriate quasimap graph space~\cite[\S7.2]{CFKM2014}. The analogous result for the Oh $I$-function is~\cite[Proposition~5.1]{Oh2016}. 

Webb has proved a `big $I$-function' version of the Abelian/non-Abelian Correspondence for target spaces that are GIT quotients of vector spaces~\cite{Webb2018}, and this immediately implies Conjectures~\ref{AnA family} and~\ref{AnA bundles family}.

\begin{proposition}
    Conjecture~\ref{AnA family} holds when $A$ is a vector space and $G$ acts on $A$ via a representation $G \mapsto \GL(A)$.
\end{proposition}

\begin{proof}
    Combining \cite[Corollary~6.3.1]{Webb2018} with \cite[Theorem~3.3]{CFK2016} shows that there are big $I$-functions $t \mapsto I_{A \GIT T}(t)$ and $t \mapsto I_{A \GIT G}(t)$ such that $I_{A \GIT T}(t) \in \cL_{A \GIT T}$ and $I_{A \GIT G}(t) \in \cL_{A \GIT G}$. Furthermore it is clear from \cite[equation 62]{Webb2018} that the Givental--Martin modification of the Weyl-invariant part of $t \mapsto I_{A \GIT T}(t)$ is $t \mapsto I_{A \GIT G}(t)$. Now argue as above.
\end{proof}

\subsection*{Connection to Earlier Work} Our formulation of the Abelian/non-Abelian Correspondence very roughly says that, for genus-zero Gromov--Witten theory, passing from an Abelian quotient $A \GIT T$ to the corresponding non-Abelian quotient $A \GIT G$ is almost the same as twisting by the non-convex bundle $\Phi \to A \GIT T$ defined by the roots of $G$. This idea goes back to the earliest work on the subject, by Bertram--Ciocan-Fontanine--Kim, and indeed our Conjecture is very much in the spirit of the discussion in~\cite[\S4]{BCFK2008}. These ideas were given a precise form in~\cite{CFKS2008}, in terms of Frobenius manifolds and Saito's period mapping; the main difference with the approach that we take here is that in~\cite{CFKS2008} the authors realise the cohomology $H^\bullet(A \GIT G)$ as the Weyl-anti-invariant subalgebra of the cohomology of the Abelian quotient $A \GIT T$, whereas we realise it as a quotient of the Weyl-invariant part of $H^\bullet(A\GIT T)$. The latter approach seems to fit better with Givental's formalism.

Ruan was the first to realise that there is a close connection between quantum cohomology (or more generally Gromov--Witten theory) and birational geometry~\cite{Ruan1999}, and the change in Gromov--Witten invariants under blow-up forms an important testing ground for these ideas. Despite the importance of the topic, however, Gromov--Witten invariants of blow-ups have been understood in rather few situations. Early work here focussed on blow-ups in points, and on exploiting structural properties of quantum cohomology such as the WDVV equations and Reconstruction Theorems~\cite{Gathmann1996, GottschePandharipande1998, Gathmann2001}. Subsequent approaches used symplectic methods pioneered by Li--Ruan~\cite{LiRuan2001,HuLiRuan2008,Hu2000,Hu2001}, or the Degeneration Formula following Maulik--Pandharipande~\cite{MaulikPandharipande2006,HeHuKeQi2018,ChenDuWang2020}, or a direct analysis of the moduli spaces involved and virtual birationality arguments~\cite{Manolache2012,Lai2009,AbramovichWise2018}. In each case the aim was to prove `birational invariance': that certain specific Gromov--Witten invariants remain invariant under blow-up. We take a different approach. Rather than deform the target space, or study the geometry of moduli spaces of stable maps explicitly, we give an elementary construction of the blow-up $\tilde{X} \to X$ in terms that are compatible with modern tools for computing Gromov--Witten invariants, and extend these tools so that they cover the cases we need. This idea -- of reworking classical constructions in birational geometry to make them amenable to computations using Givental formalism -- was pioneered in \cite{CCGK16}, and indeed Lemma~E.1 there gives the codimension-two case of our Theorem~\ref{step one}.

Compared to explicit invariance statements
\[
    \langle \pi^*\phi_{i_1}, \ldots, \pi^*\phi_{i_n} \rangle^{\tilde{X}}_{0,n,\pi^! \beta} =\langle \phi_{i_1}, \ldots, \phi_{i_n} \rangle^{X}_{0,n,\beta}
\]
as in \cite[Theorem 1.4]{Lai2009},
we pay a price for our increased abstraction: the range of invariants for which we can extract closed-form expressions is different (see Corollary~\ref{I=J}) and in general does not overlap with Lai's. But we also gain a lot by taking a more structural approach: our results determine, via a Birkhoff factorization procedure as in~\cite{CoatesGivental2007, CFK2014}, genus-zero Gromov--Witten invariants of the blow-up $\tilde{X}$ for curves of arbitrary degree (not just proper transforms of curves in the base) and with a wide range of insertions that can include gravitional descendant classes. See Remark~\ref{what can we compute blow up}. Furthermore in general one should not expect Gromov--Witten invariants to remain invariant under blow-ups. The correct statement -- cf.~Ruan's Crepant Resolution Conjecture~\cite{CCIT2009WallCrossings, CoatesRuan, Iritani2008, Iritani2009} and its generalisation by Iritani~\cite{Iritani2020} -- is believed to involve analytic continuation of Givental cones, and we hope that our formulation here will be a step towards this.

After the first version of this paper appeared on the arXiv, Fenglong You pointed us to the work~\cite{LeeLinWang2017} in which Lee, Lin, and Wang sketch a construction of blow-ups that is very similar to Theorem~\ref{step one}, and use this to compute Gromov--Witten invariants of blow-ups in complete intersections. The methods they use are different: they rely on a very interesting extension of the Quantum Lefschetz theorem to certain non-split bundles, which they will prove in forthcoming work~\cite{LeeLinWangForthcoming}. At first sight, their result~\cite[Theorem 5.1]{LeeLinWang2017} is both more general and less explicit than our results. 
In fact, we believe neither is true. Their theorem as stated applies to blow-ups in complete intersections defined by arbitrary line bundles whereas we require these line bundles to be convex; however, discussions with the authors suggest that both results apply under the same conditions, and the convexity hypothesis was omitted from~\cite[Theorem~5.1]{LeeLinWang2017} in error. Furthermore, Lee, Lin, and Wang extract genus-zero Gromov--Witten invariants by combining their generalised Quantum Lefschetz theorem with an inexplicit Birkhoff factorisation procedure whereas we use the formalism of Givental cones. We believe, though, that one can rephrase their argument entirely in terms of Givental's formalism, and after doing so their results become explicit in exactly the same range as ours. The explicit formulas are different, however, and it would be interesting to see if one can derive non-trivial identities from this. 
Note that Proposition~\ref{geometricconstruction} below is more general than the construction in~\cite[Section~5]{LeeLinWang2017}: the fact that we consider Grassmann bundles rather than projective bundles allows us to treat blow-ups in certain degeneracy loci. Combining this with the methods in Section~\ref{examples} allows one to compute genus-zero Gromov--Witten invariants of blow-ups in such degeneracy loci.

 One of the most striking features of Givental's formalism is that relationships between higher-genus Gromov--Witten invariants of different spaces can often be expressed as the quantisation, in a precise sense, of the corresponding relationship between the Lagrangian cones that encode genus-zero invariants~\cite{Givental2004}. Our version of the Abelian/non-Abelian Correspondence hints, therefore, at a higher-genus generalisation. It would be very interesting to develop and prove a higher-genus analog of Conjecture~\ref{AnA}.

\subsection*{Acknowledgements} TC was supported by ERC Consolidator Grant 682603 and  EPSRC Programme Grant EP/N03189X/1. WL and QS were supported by the EPSRC Centre for Doctoral Training in Geometry and Number Theory at the Interface, grant number EP/L015234/1. 
We thank Rachel Webb for helpful comments on an earlier draft, and Fenglong You for pointing us to~\cite{LeeLinWang2017}. TC thanks Ionu\c{t} Ciocan-Fontanine, Alessio Corti, Elana Kalashnikov, and Yuan-Pin Lee for a number of enlightening discussions.

\section{GIT Quotients and Flag Bundles}

\subsection{The topology of quotients by a non-Abelian group and its maximal torus}\label{topology}

Let $G$ be a complex reductive group acting on a smooth quasi-projective variety $A$ with polarisation given by a linearised ample line bundle $L$. Let $T \subset G$ be a maximal torus. One can then form the GIT-quotients $A \GIT G$ and $A \GIT T$. We will assume that the stable and semistable points with respect to these linearisations coincide, and that all the isotropy groups of the stable points are trivial; this ensures that the quotients $A \GIT G$ and $A \GIT T$ are smooth projective varieties. The Abelian/non-Abelian Correspondence \cite{CFKS2008} relates the genus zero Gromov--Witten invariants of these two quotients. Let $A^{s}(G)$, and respectively $A^s(T)$, denote the subsets of $A$ consisting of points that are stable for the action of $G$, and respectively $T$. The two geometric quotients $A \GIT G$ and $A \GIT T$ fit into a diagram
\begin{equation}
	\label{vbongit}
	\begin{tikzcd}
		A \GIT T & A^{s}(G)/T \arrow[d, "q"] \arrow[l, "j"', hook'] \\
		& A \GIT G                                               
	\end{tikzcd}
\end{equation}
where $j$ is the natural inclusion and $\pi$ the natural projection.

A representation $\rho \colon G \to \GL(V)$ induces a vector bundle $V(\rho)$ on $A \GIT G$ with fiber $V$. Explicitly, $V(\rho)=(A\times V)\GIT G$ where $G$ acts as  
\[
g\colon (a,v) \mapsto (ag, \rho(g^{-1}) v).
\]
Similarly, the restriction $\rho|_T$ of the representation $\rho$ induces a vector bundle $V(\rho|_T)$ over $A \GIT T$. Note that since $T$ is Abelian, $V(\rho|_T)$ splits as a direct sum of line bundles, $V(\rho|_T)=L_1 \oplus \dots \oplus L_k$ 
These bundles satisfy 
\begin{equation}\label{(21)}
	j^*V({\rho |_T}) \cong q^*V(\rho).
\end{equation}
When the representation $\rho\colon G \rightarrow \GL(V)$ is clear from context, we will suppress it from the notation, writing $V^G$ for $V(\rho)$ and $V^T$ for $V(\rho|_T)$. 

We will now describe the relationship between the cohomology rings of $A \GIT G$ and $A \GIT T$, following \cite{Martin2000}. Let $W$ be the Weyl group of $G$. $W$ acts on $A \GIT T$ and hence on the cohomology ring $H^\bullet(A \GIT T)$. Restricting the adjoint representation $\rho \colon G \to \GL(\mathfrak{g})$ to $T$, we obtain a splitting $\rho|_T=\oplus_{\alpha} \rho_\alpha$ into $1$-dimensional representations, i.e.~characters, of $T$. The set $\Delta$ of characters appearing in this decomposition is the set of roots of $G$, and forms a root system. Write $L_\alpha$ for the line bundle on $A \GIT T$ corresponding to a root $\alpha$. Fix a set of positive roots $\Phi^+$ and define 
\[
	\omega=\prod_{\alpha \in \Phi^+} c_1(L_\alpha).
\]
\begin{theorem}[Martin]
	\label{thm:Martin}
	There is a natural ring homomorphism
	\[
		H^{\bullet}(A \GIT G) \cong\frac{ H^\bullet(A \GIT T)^W}{ \Ann(\omega)}
	\]
	under which $x \in H^\bullet(A \GIT G)$ maps to $\tilde{x} \in H^\bullet(A \GIT T)$ if and only if $q^*x=j^*\tilde{x}$.
\end{theorem}
\noindent Theorem~\ref{thm:Martin} shows that any cohomology class $\tilde{x}\in H^\bullet(A \GIT T)^W$ is a lift of a class ${x} \in H^\bullet(A \GIT G)$, with $\tilde{x}$ unique up to an element of $\mathrm{Ann}(\omega)$.

\begin{assumption}
Throughout this paper, we will assume that the $G$-unstable locus $A \setminus A^s(G)$ has codimension at least $2$.
\end{assumption}

This implies that elements of $H^2(A \GIT G)$ can be lifted uniquely:
\begin{proposition}\label{maponmori}
Pullback via $q$ gives an isomorphism $H^2(A \GIT G) \cong H^2(A \GIT T)^W$, and induces a map $\varrho \colon \NE(A \GIT T) \rightarrow \NE (A \GIT G)$ where $\NE$ denotes the Mori cone. 
\end{proposition}
\begin{proof}
The assumption that $A \setminus A^s(G)$ has codimension at least $2$ implies that $A^s(T)/T \setminus A^s(G)/T$ has codimension at least $2$, so $j$ induces an isomorphism $\Pic(A^s(G)/T) \cong \Pic(A^s(T)/T)$. This gives an isomorphism $H^2(A^s(G)/T) \cong H^2(A^s(T)/T)$ since the cycle class map is an isomorphism for both spaces. Since $q^*$ always induces an isomorphism between $H^2(A \GIT G)$ and $H^2(A^s(G)/T)^W$ \cite{Borel1953}, the first claim follows. Consequently, the lifting of divisor classes is unique and can be identified with the pullback map $q^* \colon \Pic(A \GIT G) \rightarrow \Pic(A^s(G)/T)$. Since the pullback of a nef divisor class along a proper map is nef, we obtain by duality a map $\varrho: \NE(A\GIT T) \rightarrow \NE(A \GIT G)$.
\end{proof}
\begin{definition}
We say that $\tilde{\beta} \in \NE(A\GIT T)$ lifts $\beta \in \NE(A \GIT G)$ if $\varrho(\tilde{\beta}) = \beta$. Note that any effective $\beta$ has finitely many lifts. 
\end{definition}
\subsection{Partial flag varieties and partial flag bundles}\label{flag}
\subsubsection{Notation}\label{notation}
We will now specialise to the case of flag bundles and introduce notation used in the rest of the paper. Fix once and for all:
\label{setup}
\begin{itemize}
\item  a positive integer $n$ and a sequence of positive integers $r_1 < \dots < r_{\ell} < r_{\ell+1}=n$;
\item a vector bundle $E \rightarrow X$ of rank $n$ on a smooth projective variety $X$ which splits as a direct sum of line bundles $E=L_1 \oplus \dots \oplus L_n$.
\end{itemize}
We write $\Fl$ for the partial flag manifold $\Fl(r_1, \dots, r_{\ell};n)$, and $\Fl(E)$ for the partial flag bundle $\Fl(r_1, \dots, r_{\ell};E)$. 

Set $N=\sum_{i=1}^\ell r_i r_{i+1}$ and $R=r_1+\dots+r_\ell$
It will be convenient to use the indexing $\{(1,1), \dots (1, r_1), (2, 1), \dots, (\ell, r_\ell)\}$ for the set of positive integers smaller or equal than $R$.  

\subsubsection{Partial flag varieties and partial flag bundles as GIT quotients}
The partial flag manifold $\Fl$ arises as a GIT quotient, as follows. Consider $\CC^N$ as the space of homomorphisms
\begin{equation}
	\label{eq:hom}
	\bigoplus_{i=1}^\ell \Hom\left(\CC^{r_{i}}, \CC^{r_{i+1}}\right).
\end{equation}
The group $G = \prod_{i=1}^{\ell} \mathrm{GL}_{r_i}(\CC)$ acts on $\CC^N$ by 
\[
	(g_1, \dots, g_{\ell}) \cdot (A_1, \dots, A_{\ell}) = (g_2^{-1}  A_1 g_1, \dots,g_{\ell}^{-1}A_{\ell -1}g_{\ell-1}, A_{\ell}g_\ell).
\] 
Let $\rho_i \colon  G \rightarrow \GL_{r_i}(\CC)$ be the representation which is the identity on the $i$th factor and trivial on all other factors. Choosing the linearisation $\chi=\bigotimes_{i=1}^{\ell} \det(\rho_i)$,
we have that $\CC^N \GIT_\chi G$ is the partial flag manifold $\Fl$.
More generally, the partial flag bundle also arises as a GIT quotient, of the total space of the bundle of homomorphisms
\begin{equation}
	\label{eq:hom_bundle}
	\bigoplus_{i=1}^{\ell-1} \Hom\left(\cO^{\oplus r_{i}}, \cO^{\oplus r_{i+1}} \right)
	\oplus \Hom \left(\cO^{\oplus r_{\ell}},E \right)
\end{equation}
with respect to the same group $G$ and the same linearisation. $\Fl(E)$ carries $\ell$ tautological bundles of ranks $r_1, \dots, r_{\ell}$,  which we will denote $S_1, \dots, S_{\ell}$. These bundles restrict to the usual tautological bundles on $\Fl$ on each fibre. The bundle $S_i$ is induced by the representation~$\rho_i$.
\begin{definition}
Let 
\[
	p_i(t)=t^{r_i}-c_1(S_i)t^{r_i-1}+\dots + (-1)^{r_i} c_{r_i}(S_i)
\] 
be the Chern polynomial of $S_i^\vee$. We denote the roots of $p_i$ by $H_{i,j}$,~$1 \leq j \leq r_i$. The $H_{i,j}$ are in general only defined over an appropriate ring extension of $H^\bullet(\Fl(E), \CC)$, but symmetric polynomials in the $H_{i,j}$ give well-defined elements of $H^\bullet(\Fl(E), \CC)$.
\end{definition}
The maximal torus $T\subset G$ is isomorphic to $(\CC^\times)^{R}$. The corresponding Abelian quotient
\[
	\Fl(E)_T \coloneqq \Hom\big(\cdots\big) \GIT_\chi (\CC^\times)^{R},
\]
where $\Hom\big(\cdots\big)$ is the bundle of homomorphisms \eqref{eq:hom_bundle}, is a fibre bundle over $X$ with general fibre isomorphic to the toric variety $\Fl_T:= \CC^N \GIT_\chi (\CC^\times)^R$. The space $\Fl(E)_T$ also carries natural cohomology classes:
\begin{definition}
Let $\rho_{i,j}\colon  (\CC^\times)^{R} \rightarrow \GL_1(\CC)$ be the dual of the one-dimensional representation of $(\CC^\times)^{R}$ given by projection to the $(i,j)$th factor $\CC^{\times} = \GL_1(\CC)$; here we use the indexing of the set $\{1,2,\ldots,R\}$ specified in \S\ref{notation}. We define ${L}_{i,j} \in H^2(\Fl_T,\CC)$ to be the line bundle on $\Fl(E)_T$ induced by $\rho_{i,j}$ and denote its first Chern class by $\tilde{H}_{i,j}$. Similarly, we define $h_{i,j}$ to be the first Chern class of the line bundle on $\Fl_T$ induced by the represenation $\rho_{i,j}$. Equivalently, $h_{i,j}$ is the restriction of $\tilde{H}_{i,j}$ to a general fibre $\Fl_T$ of $\Fl(E)_T$. 
\end{definition}
Recall that, for a representation $\rho$ of $G$, the corresponding vector bundle $V^T$ splits as a direct sum of line bundles $F_1 \oplus \cdots \oplus F_k$. It is a general fact that if $f$ is a symmetric polynomial in the $c_1(F_i)$, then $f$ can be written as a polynomial in the elementary symmetric polynomials $e_r(c_1(F_1), \dots, c_1(F_k))$, that is, in the Chern classes $c_r(V^T)$. By \eqref{(21)} we have that $j^*c_r(V^T)=q^*c_r(V^G)$, and so replacing any occurrence of $c_r(V^T)$ by $c_r(V^G)$ gives an expression $g \in H^\bullet(A \GIT G)$ which satisfies $q^*g=j^*f$. That is, $f$ is a lift of $g$. Applying this to the dual of the standard representation $\rho_i$ of the $i$th factor of $G$ shows that any polynomial $p$ which is symmetric in each of the sets $\tilde{H}_{i,j}$ for fixed $i$ projects to the same expression in $H^\bullet(\Fl(E))$ with any occurrence of $\tilde{H}_{i,j}$ replaced by the corresponding Chern root $H_{i,j}$.

\begin{lemma}\label{torus_invariant_divisors}
Let $(\CC^\times)^R$ act on $\CC^N$, arrange the weights for this action in an $R \times N$-matrix~$(m_{i,k})$ and
consider $E=L_1 \oplus \dots \oplus L_N \xrightarrow{\pi} X$ a direct sum of line bundles. Form the associated toric fibration $E \GIT (\CC^\times)^R$  with general fibre $\CC^N \GIT (\CC^\times)^R$ and let $h_i$ (respectively $H_i$) be the first Chern class of the line bundle on $\CC^N \GIT (\CC^\times)^R$ (respectively on $E \GIT (\CC^\times)^R$ induced by the dual of the representation which is standard on the $i$th factor of $(\CC^\times)^R$ and trivial on the other factors. 
Then 
\begin{itemize}
\item the Poincar\'e duals $u_k$ of the torus invariant divisors of the toric variety $\CC^N \GIT (\CC^\times)^R$ are:
$$u_k=\sum_{k=1}^Rm_{i,k}h_i$$
\item
 the Poincar\'e duals $U_k$ of the torus invariant divisors of the total space of the toric fibration $E \GIT (\CC^\times)^R\xrightarrow{\pi} X $ are:
$$U_k=\sum_{k=1}^Rm_{i,k}H_i+\pi^*c_1(L_k)$$
\end{itemize}
\end{lemma}

When applying Lemma~\ref{torus_invariant_divisors} to our situation \eqref{eq:hom_bundle} it will be convenient to define $H_{\ell+1, j}:=\pi^*c_1(L_j^\vee)$. Then the set of torus invariant divisors is 
\begin{align*}
H_{i,j} - H_{i+1,j'} && 1 \leq i \leq \ell, \, 1 \leq j \leq r_{i}, \, 1 \leq j' \leq r_{i+1}
\end{align*}

We will also need to know about the ample cone of a toric variety $\CC^N \GIT (\CC^\times)^R$. This is most easily described in terms of the secondary fan, that is, by the wall-and-chamber decomposition of $\Pic(\CC^N \GIT (\CC^\times)^R) \otimes \RR \cong \RR^R$ given by the cones spanned by size $R-1$ subsets of columns of the weight matrix. The ample cone of $\CC^N \GIT (\CC^\times)^R$ is then the chamber that contains the stability condition $\chi$. Moreover, for a subset $\alpha \subset \{1, \dots, N\}$ of size $R$ the cone in the secondary fan spanned by the classes $u_k$, $k\in \alpha$, contains the stability condition (and therefore also the ample cone) iff the intersection $u_\alpha=\bigcap_{k \notin \alpha} u_k$ is nonempty. In this case, $U_\alpha=\bigcap_{k \notin \alpha} U_k$ restricts to a torus fixed point on every fibre and, since $E$ splits as a direct sum of line bundles, $U_\alpha$ is the image of a section of the toric fibration $\pi$. We denote this section by $s_\alpha$. By construction, the torus invariant divisors $U_{k}$,~$k \in \alpha$, do not meet $U_\alpha$, so that $s_\alpha^*(U_k)=0$ for all $k \in \alpha$.
For the toric variety $\Fl_T$ one can easily write down the set of $R$-dimensional cones containing $\chi=(1, \dots, 1)$. 
For each index $(i,j)$, choose some $j' \in \{1, \dots, r_{\ell+1}\}$. Then the cone spanned by 
\begin{align}\label{ampleconegen}
	h_{i,j}-h_{i+1,j'} && 1 \leq i < \ell-1, \, 1 \leq j \leq r_i && h_{\ell, j}, \, 1 \leq j \leq r_\ell
\end{align}
contains $\chi$ and every cone containing $\chi$ is of that form. 

\section{Givental's Formalism} \label{givental formalism}

In this section we review Givental's geometric formalism for Gromov--Witten theory, concentrating on the genus-zero case. The main reference for this is \cite{Givental2004}. Let $Y$ be a smooth projective variety and consider 
$$\cH_Y=H^\bullet(Y, \Lambda)[z, z^{-1}] \! ]=\Big\{ \sum_{k= -\infty}^{m} a_i z^i \colon  \text{$a_i \in H^{\bullet}(Y,\Lambda)$, $m \in \ZZ$}\Big\}$$ where $z$ is an indeterminate and $\Lambda$ is the Novikov ring for $Y$.
After picking a basis $\{\phi_1, \dots, \phi_N\}$ for $H^\bullet(Y;\CC)$ with $\phi_1 = 1$ and writing $\{\phi^1, \dots, \phi^N\}$ for the Poincare dual basis, we can write elements of $\cH_Y$ as 
\begin{align} 
\sum_{i=0}^m \sum_{\alpha=1}^N q_i^{\alpha}\phi_{\alpha}z^i+ \sum_{i=0}^\infty \sum_{\alpha=1}^N p_{i,\alpha}\phi^{\alpha}(-z)^{-1-i} \label{eq11}
\end{align}
where $q_i^{\alpha}$,~$p_{i,\alpha} \in \Lambda$. The $q_i^{\alpha}$,~$p_{i,\alpha}$ then provide coordinates on $\cH_Y$. 
The space $\cH_Y$ carries a symplectic form 
\begin{align*}
\Omega\colon \cH_Y \otimes \cH_Y &\rightarrow \Lambda\\
f \otimes g &\rightarrow \text{Res}_{z=0}(f(-z), g(z)) \, dz
\end{align*}
where $(\cdot , \cdot )$ denotes the Poincar\'e pairing, extended $\CC[z, z^{-1}]\!]$-linearly to $\cH_Y$. 
By construction, $\Omega$ is in Darboux form with respect to our coordinates: 
\[
	\Omega=\sum_i \sum_\alpha dp_{i,\alpha} \wedge dq_i^{\alpha}
\]
We fix a Lagrangian polarisation of $\cH$ as $\cH_Y=\cH_+ \oplus \cH_-$, where $$\cH_+=H^\bullet(Y;\Lambda)[z], \quad \cH_-=z^{-1}H^\bullet(Y;\Lambda)[\![z^{-1}]\!]$$
This polarisation $\cH_Y = \cH_+ \oplus \cH_-$ identifies $\cH_Y$ with $T^* \cH_+$. We now relate this to Gromov--Witten theory.

\begin{definition} The \emph{genus-zero descendant potential} is a generating function for genus-zero Gromov--Witten invariants:
	$$\mathcal{F}_{Y}^{0} = \sum_{n = 0}^\infty \sum_{d \in \NE(Y)} \frac{Q^d}{n!} t^{\alpha_1}_{i_1} \dots t^{\alpha_{n}}_{i_n}\langle \phi_{\alpha_1}\psi^{i_1}, \dots, \phi_{\alpha_n}\psi^{i_n} \rangle_{0,n,d}$$ Here $t_i^\alpha$ is a formal variable, $\NE(Y)$ denotes the Mori cone of $Y$, and Einstein summation is used for repeated lower and upper indices. 
\end{definition}

\noindent After setting 
\begin{equation}
	\label{eq:dilaton}
	t^{\alpha}_{i} = q^{\alpha}_{i} + \delta^{i}_{1}\delta^{1}_{\alpha},
\end{equation}
where $\delta_i^j$ denotes the Kronecker delta, we obtain a (formal germ of a) function $\mathcal{F}^{0}_Y \colon  \cH_+ \rightarrow \Lambda$. 
\begin{definition}
The Givental cone $\cL_Y$ of $Y$ is the graph of the differential of $\mathcal{F}_{Y}^{0}\colon  \cH_+ \rightarrow \Lambda$:
	$$\cL_Y = \left\{(\mathbf{q,p}) \in T^*\cH_Y= \cH_+ \oplus \cH_- \colon  p_{i, \alpha} = \frac{\partial \mathcal{F}^0_{Y}}{\partial q^{\alpha}_i}\right\}$$
	Note that $\cL_Y$ is Lagrangian by virtue of being the graph of the differential of a function. Moreover, it has the following special geometric properties \cite{Givental2004, CCIT2009Computing, CoatesGivental2007}
	\begin{itemize}
	    \item $\cL$ is preserved by scalar multiplication, i.e. it is (the formal germ of) a cone
	    \item the tangent space $T_f$ of $\cL_Y$ at $f \in \cL_Y$ is tangent to $\cL$ exactly along $z T_f$. This means:
	    \begin{enumerate}
	        \item $zT_f \subset \cL_Y$
	        \item for $g \in zT_f$, we have $T_g = T_f$
	        \item $T_f \cap \cL_Y = z T_f$
	    \end{enumerate}
	\end{itemize}
\end{definition}\label{J}
A general point of $\cL_Y$ can be written, in view of the dilaton shift \eqref{eq:dilaton}, as
\begin{align*}
	&{-z} + \sum_{i = 0}^\infty t^{\alpha}_i \phi_{\alpha}z^i + \sum_{n = 0}^\infty \sum_{d \in \NE(Y)} \frac{Q^d}{n!} t^{\alpha_1}_{i_1} \dots t^{\alpha_{n}}_{i_n}\langle \phi_{\alpha_1}\psi^{i_1}, \dots, \phi_{\alpha_n}\psi^{i_n}, \phi_{\alpha}\psi^{i} \rangle_{0,n+1,d} \phi^{\alpha}(-z)^{-i-1} \\
	= &{-z} + \sum_{i = 0}^\infty t^{\alpha}_i \phi_{\alpha}z^i + \sum_{n = 0}^\infty \sum_{d \in \NE(Y)} \frac{Q^d}{n!} t^{\alpha_1}_{i_1} \dots t^{\alpha_{n}}_{i_n}\langle \phi_{\alpha_1}\psi^{i_1}, \dots, \phi_{\alpha_n}\psi^{i_n}, \frac{\phi_{\alpha}}{-z - \psi} \rangle_{0,n+1,d} \phi^{\alpha}
\end{align*}
Thus knowing $\cL_Y$ is equivalent to knowing all genus-zero Gromov--Witten invariants of $Y$. Setting $t_k^\alpha=0$ for all $k>0$, we obtain the \emph{$J$-function} of $
Y$:
\begin{equation*}
	J(\tau,-z) = -z + \tau + \sum_{n = 0}^\infty \sum_{d \in \mathrm{\NE(X)}} \frac{Q^d}{n!} \left\langle \tau, \dots \tau, \frac{\phi_{\alpha}}{-z - \psi} \right\rangle_{0,n+1,d} \phi^{\alpha}
\end{equation*}
where $\tau = t^1_0 \phi_1 + \dots t^N_0 \phi_N \in H^\bullet(Y)$. The $J$-function is the unique family of elements $\tau \mapsto J(\tau,-z)$ on the Lagrangian cone such that 
\begin{equation*}
	J(\tau, -z) = -z + \tau + O(z^{-1}).
\end{equation*} 

We will need a generalisation of all of this to twisted Gromov--Witten invariants~\cite{CoatesGivental2007}. Let~$F$ be a vector bundle on $Y$ and consider the universal family over the moduli space of stable maps
$$\begin{tikzcd}
{C_{0,n,d}} \arrow[d, "\pi"'] \arrow[r, "f"] & Y \\
{Y_{0,n,d}}                                  &  
\end{tikzcd}$$
Let $\pi_!$ be the pushforward in $K$-theory. We define $$F_{0,n,d} = \pi_{!}f^*F=R^{0}\pi_{*}f^*F-R^{1}\pi_{*}f^*F$$
(the higher derived functors vanish).
In general $F_{0,n,d}$ is a class in $K$-theory and not an honest vector bundle. This means that in order to evaluate a characteristic class $\mathbf{c}(\cdot)$ on $F_{0,n,d}$ we need $\mathbf{c}(\cdot)$ to be \emph{multiplicative} and \emph{invertible}. We can then set 
\[
	\mathbf{c}(F_{0,n,d}) = \mathbf{c}(R^{0}\pi_{*}f^* F) \cup \mathbf{c}(R^{1}\pi_{*}f^* F)^{-1}
\] 
where $\mathbf{c}(R^{i}\pi_{*}f^* F)$ is defined using an appropriate locally free resolution. 
\begin{definition} \label{equivarianteuler}
Let $F$ be a vector bundle on $Y$ and let $\mathbf{c}(\cdot)$ be an invertible multiplicative characteristic class. We will refer to the pair $(F, {\bf c})$ as twisting data. Define $(F, {\bf c})$-twisted Gromov--Witten invariants as
	\begin{equation*}
		\langle \alpha_1 \psi_1^{i_1}, \dots \alpha_n \psi_{n}^{i_n} \rangle_{0,n,d}^{F, {\bf c}} = \int_{[Y_{0,n,d}]^{\mathrm{vir}} \cap \mathbf{c}(F_{0,n,d})} \ev_1^{*}\alpha_1 \cup \dots \cup \ev_n^{*}\alpha_n \cup \psi_1^{i_1} \cup \dots \cup \psi_{n}^{i_n}
	\end{equation*}
\end{definition}
Any multiplicative invertible characteristic class can be written as $\mathbf{c}(\cdot ) = \exp(\sum_{k \geq 0} s_k \ch_k(\cdot))$, where $\ch_k$ is the $k$th component of the Chern character and $s_0$,~$s_1$,~\ldots are appropriate coefficients. So we work with cohomology groups $H^{\bullet}(X, \Lambda_s)$, where $\Lambda_s$ is the completion  of $\Lambda[s_0, s_1, \dots]$ with respect to the valuation
\begin{equation*}
	v(Q^d) = \big\langle c_1(\cO(1)), d \big \rangle, \quad v(s_k) = k+1.
\end{equation*}
Most of the definitions from before now carry over. 
We have the twisted Poincar\'e pairing $(\alpha,\beta)^{F, {\bf c}} = \int_Y \mathbf{c}(F) \cup \alpha \cup \beta $ which defines the basis $\phi^1, \dots \phi^N$ dual to our chosen basis $1 =\phi_1, \dots, \phi_N$ for $H^\bullet(Y)$. The Givental space becomes $\cH_Y = H^{\bullet}(Y,\Lambda_s) \, \otimes \, \CC[z,z^{-1}]\!]$ with the twisted symplectic form $$\Omega^{F, {\bf c}}(f(z), g(z)) = \mathrm{Res}_{z=0}\big(f(-z),g(z)\big)^{F, {\bf c}}dz.$$ This form admits Darboux coordinates as before which give a Lagrangian polarisation of $\cH_Y$. Then the twisted Lagrangian cone $\cL_{F, {\bf c}}$ is defined, via the dilaton shift \eqref{eq:dilaton}, as the graph of the differential of the generating function $\mathcal{F}^{0,F, {\bf c}}_Y$ for genus zero \textit{twisted} Gromov--Witten invariants. Finally, just as before, we can define a twisted $J$-function:
\begin{definition} \label{twisted J}
Given twisting data $(F, {\bf c})$ for $Y$, the twisted $J$-function is:
\begin{equation*}
	J_{F, {\bf c}}(\tau,{-z}) = {-z} + \tau + \sum_{n = 0}^\infty \sum_{d \in \NE(Y)} \frac{Q^d}{n!} \left\langle \tau, \dots \tau, \frac{\phi_{\alpha}}{-z - \psi} \right\rangle^{F, {\bf c}}_{0,n+1,d} \phi^{\alpha}
\end{equation*}
\end{definition}
\noindent This is once again characterised as the unique family $\tau \mapsto J_{F, {\bf c}}(\tau,-z)$ of elements of the twisted Lagrangian cone of the form
\begin{equation*}
	J_{F, {\bf c}}(\tau, -z) = -z + \tau + O(z^{-1})
\end{equation*} 
Note that we can recover the untwisted theory by setting $\mathbf{c}=1$.

In what follows we take $\mathbf{c}$ to be the $\CC^\times$-equivariant Euler class \eqref{intro equivariant Euler}, which is multiplicative and invertible. The $\Cstar$-action here is the canonical $\CC^\times$-action on any vector bundle given by rescaling the fibres. We write $F_\lambda$ for the twisting data $(F, \mathbf{c})$, where $F$ is equipped with the $\CC^\times$-action given by rescaling the fibres with equivariant parameter $\lambda$. In this setting, Gromov--Witten invariants (and the coefficients $s_k$) take values in the fraction field $\CC(\lambda)$ of the $\CC^\times$-equivariant cohomology of a point. Here $\lambda$ is the hyperplane class on $\mathbb{CP}^{\infty}$, so that $H^{\bullet}_{\CC^\times}(\{\mathrm{pt}\}) = \CC[\lambda ]$, and we work over the field $\CC(\lambda)$. 

\begin{remark}
	As we have set things up, the twisted cone $\cL_{F_\lambda}$ is a Lagrangian submanifold of the symplectic vector space $\big(\cH_Y, \Omega^{F_\lambda}\big)$, so as $\lambda$ varies both the Lagrangian submanifold and the ambient symplectic space change. To obtain the picture described in the Introduction, where all the Lagrangian submanifolds $\cL_{F_\lambda}$ lie in a single symplectic vector space $\big(\cH_Y, \Omega \big)$, one can identify $\big(\cH_Y, \Omega \big)$ with $\big(\cH_Y, \Omega^{F_\lambda} \big)$ by multiplication by the square root of the equivariant Euler class of $F$. See~\cite[\S8]{CoatesGivental2007} for details.
\end{remark}

\subsection{Twisting the $I$-function}\label{I-functions}
We will now prove a general result following an argument from \cite{CCIT2009Computing}. We say that a family $\tau \mapsto I(\tau)$ of elements of $\cH_Y$ \emph{satisfies the Divisor Equation} if the parameter domain for $\tau$ is a product $U \times H^2(Y)$ and $I(\tau)$ takes the form
$$ I(\tau) = \sum_{\beta \in \NE(Y)}Q^{\beta} I_{\beta}(\tau,z) $$
where
\begin{align}\label{divisor equation}
	z\nabla_{\rho}I_{\beta}= \big(\rho + \langle \rho,\beta \rangle z\big) I_{\beta}
	&&
	\text{for all $\rho \in H^2(Y)$.}
\end{align}
Here $\nabla_\rho$ is the directional derivative along $\rho$. Let $F'$ be a vector bundle on $Y$, and consider any family $\tau \mapsto I(\tau) \in \cL_{F'_\mu}$ that satisfies the Divisor Equation. Given another vector bundle $F$ which splits as a direct sum of line bundles $F=F_1 \oplus \dots \oplus F_k$, we explain how to modify the family $\tau \mapsto I(\tau)$ by introducing explicit hypergeometric factors that depend on $F$. We prove that (1) this modified family can be written in terms of the {\it Quantum Riemann-Roch operator} and the original family; and (2) the modified family lies on the twisted Lagrangian cone $\cL_{F_\lambda \oplus F'_\mu}$.

\begin{definition}
Define the element $G(x,z) \in \cH_Y$ by  $$G(x,z) := \sum_{l=0}^\infty \sum_{m=0}^\infty s_{l + m - 1} \frac{B_m}{m!}\frac{x^l}{l!}z^{m-1}$$
where $B_m$ are the Bernoulli numbers and the $s_k$ are the coefficients obtained by writing the $\CC^\times$-equivariant Euler class \eqref{intro equivariant Euler} in the form $\exp\big(\sum_{k \geq 0} s_k \ch_k(\cdot)\big)$. 
\end{definition}

\begin{remark}
	The discussion in this section is valid for any invertible multiplicative characteristic class, not just the equivariant Euler class, but we will neither need nor emphasize this.
\end{remark}

\begin{definition}\label{delta}
Let $F$ be a vector bundle -- not necessarily split -- and let $f_i$ be the Chern roots of $F$. Define the \textit{Quantum Riemann-Roch operator}, $\Delta_{F_\lambda} \colon  \cH_{Y} \rightarrow \cH_{Y}$ as multiplication by 
	\begin{equation*}
	\Delta_{F_\lambda} = \prod_{i=1}^{k} \exp(G(f_i, z))
	\end{equation*}
\end{definition}
\begin{theorem}[\cite{CoatesGivental2007}] \label{QRR} \ 
$\Delta_{F_\lambda}$ gives a linear symplectomorphism of $(\cH_{Y},\Omega_{Y})$  with $(\cH_{Y}, \Omega_{Y}^{F_\lambda})$ such that $$\Delta_{F_\lambda}(\cL_{Y}) = \cL_{F_\lambda}$$
\end{theorem}

Since $\Delta_{F_\lambda} \circ \Delta_{F'_\mu}=\Delta_{F_\lambda \oplus F'_\mu}$, it follows immediately that
$$\Delta_{F_\lambda}(\cL_{F'_\mu}) = \cL_{F_\lambda \oplus F'_\mu}.$$

\begin{lemma} \label{(3)}
	Let $F$ be a vector bundle and let $f_1, \ldots, f_k$ be the Chern roots of $F$. Let 
	$$D_{F_\lambda}=\prod_{i=1}^{k} \exp\big({-G}(z\nabla_{f_i},z)\big)$$
	and suppose that $\tau \mapsto I(\tau)$ is a family of elements of $\cL_{F'_\mu}$. Then $\tau \mapsto D_{F_\lambda}(I(\tau))$ is also a family of elements of $\cL_{F'_\mu}$.
\end{lemma}
\begin{proof}
	This follows \cite[Theorem~4.6]{CCIT2009Computing}. Let $h = -z + \sum_{i=0}^m t_i z^i + \sum_{j=0}^{\infty} p_{j}(-z)^{-j-1}$ be a point on $\cH_{Y}$. The Lagrangian cone $\cL_{F'_\mu}$ is defined by the equations $E_j=0$,~$j=0,1,2,\dots$ where $$E_j(h) = p_j - \sum_{n \geq 0} \sum_{d \in \NE(Y)} \frac{Q^d}{n!} t^{\alpha_1}_{i_1} \dots t^{\alpha_{n}}_{i_n}\langle \phi_{\alpha_1}\psi^{i_1}, \dots, \phi_{\alpha_n}\psi^{i_n}, \phi_{\alpha}\psi^{j} \rangle_{0,n+1,d} \phi^{\alpha}$$
	We need to show that $E_j(D_{F_\lambda}(I)) = 0$. Note that $D_{F_\lambda}(I) = \prod_{i=1}^{k} \exp(-G(z\nabla_{f_i},z))I$ depends on the parameters $s_i$. For notational simplicity assume that $k=1$, so that $$D_{F_\lambda}(I) = \exp\big({-G}(z\nabla_{f},z)\big)I$$ Set $\deg s_i = i+1$. We will prove the result by inducting on degree. Note that if $s_0 = s_1= \dots = 0$ then $D_{F_\lambda}(I) = I$ so that $E_j(D_{F_\lambda}(I)) = 0$. Assume by induction that $E_j(D_{F_\lambda}(I))$ vanishes up to degree $n$ in the variables $s_0, s_1, s_2, \dots$ Then $$\frac{\partial}{\partial s_i} E_j(D_{F_\lambda}(I)) = d_{D_{F_\lambda}(I)}E_j (z^{-1}P_{i}(z \nabla_{f},z)D_{F_\lambda}(I))$$ where $$P_{i}(z \nabla_{f},z) = \sum_{m=0}^{i+1} \frac{1}{m!(i+1-m)!}z^{m}B_m (z \nabla_{f})^{i+1-m}$$
	
	By induction there exists $D_{F_\lambda}(I)' \in \cL_{F'_\mu}$ such that  $$\frac{\partial}{\partial s_i} E_j(D_{F_\lambda}(I)) = d_{D_{F_\lambda}(I)'}E_j (z^{-1}P_{i}(z \nabla_{f},z)D_{F_\lambda}(I)')$$ 
	up to degree $n$. But the right hand side of this expression is zero, since the term in brackets lies in the tangent space to the Lagrangian cone. Indeed, applying $\nabla_f$ to $D_{F_\lambda}(I_{Y})'$ -- or to any family lying on the cone -- takes it to the tangent space of the cone at the point. And then applying $z\nabla_f$ preserves that tangent space.
\end{proof}
\begin{corollary}\label{Ithm}
	Let $\tau \mapsto I(\tau)$ be a family of elements of $\cL_{F'_\mu}$. Then $\tau \mapsto \Delta_{F_\lambda}(D_{F_\lambda}(I(\tau)))$ is a family of elements of $\cL_{F_\lambda \oplus F'_\mu}$.
\end{corollary}
\begin{proof}
This follows immediately by combining \ref{QRR} and \ref{(3)}
\end{proof}
Corollary \ref{Ithm} produces a family of elements on the twisted Lagrangian cone $\cL_{F_\lambda \oplus F'_\mu}$, but in general it is not obvious whether the nonequivariant limit $\lambda \rightarrow 0$ of this family exists. However, in the case when $F$ is split and $\tau \mapsto I(\tau)$ satisfies the Divisor Equation we will show that the family $\Delta_{F_\lambda}(D_{F_\lambda}(I(\tau, -z)))$ is equal to the \emph{twisted $I$-function} $I_{F'_\mu \oplus {F_\lambda}}$ given in Definition~\ref{twistedI}. This has an explicit expression, which makes it easy to check whether the nonequivariant limit exists.  We make the following definitions.

\begin{definition}\label{modification} \label{twistedI}
	Let $\tau \mapsto I(\tau)$ be a family of elements of $\cL_{F'_\mu}$. Let $F= F_1 \oplus \dots \oplus F_k$ be a direct sum of line bundles, and let $f_i=c_1(F_i)$. For $\beta \in \NE(Y)$, we define the modification factor 
	$$M_{\beta}(z) = \prod_{i=1}^{k} \frac{\prod_{m=-\infty}^{\langle f_{i}, \beta \rangle} \lambda + f_{i} + mz }{\prod_{m=-\infty}^{0} \lambda + f_{i} + mz }$$
	The associated \textit{twisted $I$-function} is
	\begin{equation*}
	I^{\tw}(\tau) = \sum_{\beta \in \NE(Y)} Q^{\beta} I_{\beta}(\tau,z) \cdot M_{\beta}(z)
	\end{equation*}	
\end{definition}

To relate $M_{\beta}(z)$ to the Quantum Riemann--Roch operator we will need the following Lemma:
\begin{lemma}\label{delta-M}
	$$M_{\beta}(-z) = \Delta_{F_\lambda}\left(\prod_{i=1}^{k}\exp(- G(f_{i} - \langle f_{i}, \beta \rangle z, z))\right)$$
\end{lemma} 
\begin{proof}
	 Define $$\mathbf{s}(x) = \sum_{k \geq 0} s_k \frac{x^k}{k!}$$	 
	By \cite[equation 13]{CCIT2009Computing} we have that 
	\begin{equation}\label{gamma}
		G(x+z,z) = G(x,z) + \mathbf{s}(x)
	\end{equation}
	We can rewrite
 	$$M_{\beta}(z) = \prod_{i=1}^{k} \frac{\prod_{m=-\infty}^{\langle f_{i}, \beta \rangle} \lambda + f_{i} + mz }{\prod_{m=-\infty}^{0} \lambda + f_{i} + mz } = \prod_{i=1}^{k} \frac{\prod_{m=-\infty}^{\langle f_{i}, \beta \rangle} \exp[\mathbf{s}(f_{i} + mz)] }{\prod_{m=-\infty}^{0} \exp[\mathbf{s}(f_{i} + mz)] }$$
	and so
	\begin{align*}
	M_{\beta}(-z) =& \prod_{i=1}^{k}\exp\left(\sum_{m=-\infty}^{\langle f_i, \beta \rangle} \mathbf{s}(f_{i} - mz) - \sum_{m=-\infty}^{0} \mathbf{s}(f_{i} - mz))\right) \\
	=& \prod_{i=1}^{k} \exp(G(f_i,z) - G(f_i - \langle f_i, \beta \rangle z, z) 
	\end{align*}
	where for the second equality we used \eqref{gamma}.
\end{proof}
\begin{proposition}\label{twisted=Deltad}
Let $\tau \mapsto I(\tau)$ be a family of elements of $\cL_{F'_\mu}$ that satisfies the Divisor Equation, and let $F=F_1 \oplus \dots \oplus F_k$ be a direct sum of line bundles. Then 
\begin{equation}\label{twisted=Deltadeq}
I^{\tw}=\Delta_{F_\lambda}(D_{F_\lambda}(I)).
\end{equation}
As a consequence, $\tau \mapsto I^{\tw}(\tau)$ is a family of elements on the cone $\cL_{F_\lambda \oplus F'_\mu}$.
\end{proposition}
\begin{proof}
Lemma \ref{delta-M} shows that
\begin{equation}\label{imagebeta}
I^{\tw}(\tau) = \Delta_{F_\lambda}\left(\sum_{\beta \in \NE(Y)}\prod_{i=1}^{k} \exp(-G(f_i - \langle f_i, \beta \rangle z, z))I_{\beta}(\tau,z)\right)
\end{equation}
Applying the Divisor Equation, we can rewrite this as
\begin{equation}
I^{\tw}=\Delta_{F_\lambda}(D_{F_\lambda}(I))
\end{equation}
as required. The rest is immediate from \ref{Ithm}.
\end{proof}

\begin{proposition}\label{nonequiexists}
If the line bundles $F_i$ are nef, then the nonequivariant limit $\lambda \rightarrow 0$ of $I^{\tw}(\tau)$ exists.
\end{proposition}

\begin{proof}
	This is immediate from Definition~\ref{twistedI}.
\end{proof}

\section{The Givental--Martin cone}\label{giventalmartincones}
We now restrict to the situation described in the Introduction, where the action of a reductive Lie group $G$ on a smooth quasiprojective variety $A$ leads to smooth GIT quotients $A \GIT G$ and~$A \GIT T$.  As discussed, the roots of $G$ define a vector bundle $\Phi = \oplus_\rho L_\rho \to Y$, where $Y = A \GIT T$, and we consider twisting data $(\Phi, \mathbf{c})$ for $Y$ where $\mathbf{c}$ is the $\CC^\times$-equivariant Euler class. 
We call the modification factor in this setting the \emph{Weyl modification factor}, and denote it as
\begin{equation}\label{modg}
    W_{\beta}(z) = \prod_{\alpha} \frac{\prod_{m=-\infty}^{\langle c_{1}(L_{\alpha}),  \beta \rangle} c_{1}(L_{\alpha}) + \lambda + mz}{\prod_{m=-\infty}^{0} c_{1}(L_{\alpha}) + \lambda + mz}
\end{equation}
where the product runs over all roots $\alpha$.
For any family $\tau \mapsto I(\tau)= \sum_{\beta \in \NE(Y)}Q^{\beta}I_{\beta}(\tau,z)$ of elements of $\cH_Y$, the corresponding twisted $I$-function is 
\begin{equation} \label{general Weyl twist}
	I^{\tw}(\tau) = \sum_{\beta \in \NE(Y)}Q^{\beta}I_{\beta}(\tau,z) \cdot  W_{\beta}(z)
\end{equation}
Since the roots bundle $\Phi$ is not convex, in general the non-equivariant limit $\lambda \to 0$ of $I^{\tw}$ will not exist. Recall from \eqref{quotientH}, however, the map $p \colon \cH^W_{A \GIT T} \to \cH_{A \GIT G}$.
\begin{lemma}\label{IGMexists}
    Suppose that $I$ is Weyl-invariant. Then $p \circ I^\tw$ has a well-defined limit as $\lambda \rightarrow 0$.
\end{lemma}
\begin{proof}
	The map $p$ is given by the composition of the map on Novikov rings induced by $$\varrho \colon \NE(A \GIT T) \to \NE(A \GIT G)$$ (see Proposition~\ref{maponmori}) with the projection map $H^\bullet(A \GIT T; \CC)^W \to H^\bullet(A \GIT G; \CC)$ (see Theorem~\ref{thm:Martin}). Since $I(\tau)$ is Weyl-invariant, $I^\tw(\tau)$ is also Weyl invariant and so, after applying $\varrho$, the coefficient of each Novikov term $Q^\beta$ in $\tau \mapsto I^\tw(\tau)$ lies in $H^\bullet(A \GIT T; \CC)^W$. The composition $p \circ I^{\tw}$ is therefore well-defined.

	The Weyl modification \eqref{modg} contains many factors
	$$
	\frac{c_1(L_\alpha) + \lambda + m z}{- c_1(L_\alpha) + \lambda - m z}
	$$
	which arise by combining the terms involving roots $\alpha$ and $-\alpha$. Such factors have a well-defined limit, $-1$, as $\lambda \to 0$. Therefore the limit of $p \circ I^\tw$ as $\lambda \to 0$ is well-defined if and only if the limit of
	\begin{equation} \label{p of I intermediate}
		p \left( \sum_{\beta \in \NE(Y)}Q^{\beta}I_{\beta}(\tau,z) \cdot  (-1)^{\epsilon(\beta)}\prod_{\alpha \in \Phi^+} \frac{ c_1(L_\alpha) \pm \lambda + \langle c_1(L_\alpha), \beta \rangle z}{c_1(L_\alpha) \mp \lambda}  \right)
	\end{equation}
	as $\lambda \to 0$ is well-defined, and the two limits coincide. Here $\Phi^+$ is the set of positive roots of $G$, and $\epsilon(\beta) = \sum_{\alpha \in \Phi^+} \langle c_1(L_\alpha), \beta \rangle$; cf.~\cite[equation 3.2.1]{CFKS2008}. The limit $\lambda \to 0$ of the denominator terms 
	$$
	\prod_{\alpha \in \Phi^+} \big(c_1(L_\alpha) - \lambda\big)
	$$ 
	in \eqref{p of I intermediate} is the fundamental Weyl-anti-invariant class $\omega$ from the discussion before Theorem~\ref{thm:Martin}. Furthermore
	$$ \sum_{\beta \in \NE(Y)}Q^{\beta}I_{\beta}(\tau,z) \cdot  (-1)^{\epsilon(\beta)}\prod_{\alpha \in \Phi^+} \big(c_1(L_\alpha) + \lambda + \langle c_1(L_\alpha), \beta \rangle z\big)$$
	has a well-defined limit as $\lambda \to 0$ which, as it is Weyl-anti-invariant, is divisible by $\omega$. The quotient here is unique up to an element of $\Ann(\omega)$, and therefore the projection of the quotient along Martin's map $H^\bullet(A \GIT T; \CC)^W \to H^\bullet(A \GIT G; \CC)$ is unique. It follows that the limit as $\lambda \to 0$ of $p \circ I^\tw$ is well-defined.
\end{proof}

\begin{definition} \label{IGM definition}
	Let $\tau \mapsto I(\tau)$ be a Weyl-invariant family of elements of $\cH_Y$ and let $I^{\tw}$ denote the twisted $I$-function as above. We call the nonequivariant limit of $\tau \mapsto p \big(I^{\tw}(\tau)\big)$ the \emph{Givental--Martin modification} of the family $\tau \mapsto I(\tau)$, and denote it by $\tau \mapsto I_{\GM}(\tau)$
\end{definition}

Recall that we have fixed a representation $\rho$ of $G$ on a vector space $V$, and that this induces vector bundles $V^T \to A \GIT T$ and $V^G \to A \GIT G$. Since the bundle $\Phi \to A \GIT T$ is not convex, one cannot expect the non-equivariant limit of $\cL_{ \Phi_\lambda \oplus V^T_\mu}$ to exist. Nonetheless, the projection along \eqref{quotientH} of the Weyl-invariant part of $\cL_{ \Phi_\lambda \oplus V^T_\mu}$ does admit a non-equivariant limit.
\begin{theorem}\label{GMlimit} 
	The non-equivariant limit $\lambda \to 0$ of $p \left(\cL_{\Phi_\lambda \oplus V^T_\mu} \cap \cH^{W}_{A \GIT T} \right)$ exists.
\end{theorem}
\noindent We call this non-equivariant limit the \emph{twisted Givental--Martin cone} $\cL_{\GM, V^T_\mu} \subset \cH^{W}_{A\GIT T}$. 

\begin{proof}[Proof of Theorem~\ref{GMlimit}]
	Recall the twisted $J$-function $J_{V^T_\mu}(\tau,{-z})$ from Definition~\ref{twisted J}. By~\cite{CoatesGivental2007} a general point
	$$ {-z} + t_0 + t_1 z + \cdots + O(z^{-1}) $$
	on $\cL_{V^T_\mu}$ can be written as 
	$$ J_{V^T_\mu} \big(\tau(\bt),{-z}\big) + \sum_{\alpha=1}^N C_\alpha(\bt, z) z \frac{\partial J_{V^T_\mu}}{\partial \tau^\alpha}\big(\tau(\bt), {-z}\big)
	$$
	for some coefficients $C_\alpha(\bt, z)$ that depend polynomially on $z$ and some $H^\bullet(A \GIT T)$-valued
	function $\tau(\bt)$ of $\bt = (t_0,t_1,\ldots)$. The Weyl modification $\tau \mapsto I^\tw(\tau)$ of $\tau \mapsto J_{V^T_\mu}(\tau,-z)$ satisfies $I^\tw(\tau) \equiv J_{V^T_\mu}(\tau,{-z})$ modulo Novikov variables, and $I^\tw(\tau) \in \cL_{\Phi_\lambda \oplus V^T_\mu}$ by Proposition~\ref{twisted=Deltad}, so a general point
	\begin{equation} \label{general point on equivariant Weyl}
		{-z} + t_0 + t_1 z + \cdots + O(z^{-1})
	\end{equation}
	on $\cL_{\Phi_\lambda \oplus V^T_\mu}$ can be written as
	$$ I^\tw \big(\tau(\bt)^\dagger,{-z}\big) + \sum_{\alpha=1}^N C_\alpha(\bt, z)^\dagger z \frac{\partial I^\tw}{\partial \tau^\alpha}\big(\tau(\bt)^\dagger, {-z}\big) $$
	for some coefficients $C_\alpha(\bt, z)^\dagger$ that depend polynomially on $z$ and some $H^\bullet(A \GIT T)$-valued
	function $\tau(\bt)^\dagger$.  Since the twisted $J$-function is Weyl-invariant, so is $I^\tw(\tau)$, and thus if \eqref{general point on equivariant Weyl} is Weyl-invariant then we may take $C_\alpha(\bt, z)^\dagger$ to be such that $\sum_\alpha C_\alpha(\bt, z)^\dagger \phi_\alpha$ is Weyl-invariant. Projecting along \eqref{quotientH} we see that a general point
	\begin{equation} \label{general point on equivariant GM}
		{-z} + t_0 + t_1 z + \cdots + O(z^{-1})
	\end{equation}
	on $p \left(\cL_{\Phi_\lambda \oplus V^T_\mu} \cap \cH^{W}_{A \GIT T} \right)$ can be written as 
	$$ p \circ I^\tw \big(\tau(\bt)^\ddagger,{-z}\big) + \sum_{\alpha=1}^N C_\alpha(\bt, z)^\ddagger z \frac{\partial (p \circ I^\tw)}{\partial \tau^\alpha}\big(\tau(\bt)^\ddagger, {-z}\big) $$
	for some coefficients $C_\alpha(\bt, z)^\ddagger$ that depend polynomially on $z$ and some $H^\bullet(A \GIT T)$-valued
	function $\tau(\bt)^\ddagger$. Furthermore, since $p \circ I^\tw(\tau)$ has a well-defined non-equivariant limit $I_{\GM}(\tau)$, we see that $C_{\alpha}(\bt, z)^\ddagger$ also admits a non-equivariant limit. Hence a general point \eqref{general point on equivariant GM} on $p \left(\cL_{\Phi_\lambda \oplus V^T_\mu} \cap \cH^{W}_{A \GIT T} \right)$ has a well-defined limit as $\lambda \to 0$.
\end{proof}

\begin{corollary} \label{GMlimit no bundle}
	The non-equivariant limit $\lambda \to 0$ of $p \left(\cL_{\Phi_\lambda} \cap \cH^{W}_{A \GIT T} \right)$ exists.  
\end{corollary}

\noindent We call this non-equivariant limit the \emph{Givental--Martin cone} $\cL_{\GM} \subset \cH^{W}_{A\GIT T}$. 

\begin{proof}
	Take the vector bundle $V^T$ in Theorem~\ref{GMlimit} to have rank zero.
\end{proof}

\begin{corollary}\label{IGMonLGM}
	If $\tau \mapsto I(\tau)$ is a Weyl-invariant family of elements of $\cL_{V^T_\mu}$ that satisfies the Divisor Equation \eqref{divisor equation} then the Givental--Martin modification $\tau \mapsto I_{\GM}(\tau)$ is a family of elements of~$\cL_{\GM, V^T_\mu}$
\end{corollary}
\begin{proof}

	Proposition~\ref{twisted=Deltad} implies that 
$\tau \mapsto I^\tw(\tau, -z)$ is a family of elements on $\cL_{  \Phi_{\lambda} \oplus V^T_\mu}$. Projecting along \eqref{quotientH} and taking the limit $\lambda \rightarrow 0$, which exists by Lemma~\ref{IGMexists}, proves the result.
\end{proof}

This completes the results required to state the Abelian/non-Abelian Correspondence (Conjectures~\ref{AnA} and~\ref{AnA family}) and the Abelian/non-Abelian Correspondence with bundles (Conjectures~\ref{AnA bundles} and~\ref{AnA bundles family}).

\section{The Abelian/non-Abelian Correspondence for Flag Bundles}

\subsection{The Work of Brown and Oh}\label{brownohwn}
In this section we will review results by Brown~\cite{Brown2014} and Oh~\cite{Oh2016}, and situate their work in terms of the Abelian/non-Abelian Correspondence (Conjecture~\ref{AnA family}). In particular, we show that the Givental--Martin modification of the Brown $I$-function is the Oh $I$-function. 
We freely use the notation introduced in Section~\ref{notation}. 

Let $X$ be a smooth projective variety. We will decompose the $J$-function of $X$, defined in~\S\ref{J}, into contributions from different degrees:
\begin{equation} \label{JX by degrees}
	J_{X}(\tau,z)= \sum_{D \in \NE(X)} J_X^{D}(\tau,z) Q^{D}.
\end{equation}
Recall that we have a direct sum of line bundles $E = L_1 \oplus \dots \oplus L_n \xrightarrow{\pi} X$, and that $\Fl(E) = \Fl(r_1, \dots, r_\ell,E) = A \GIT G$ is the partial flag bundle associated to $E$. As in \S\ref{flag}, we form the toric fibration $\Fl(E)_T = A \GIT T$ with general fibre $\CC^N \GIT (\CC^\times)^R$. We denote both projection maps $\Fl(E) \to X$ and $\Fl(E)_T \to X$ by $\pi$. For the sake of clarity, we will denote homology and cohomology classes on $\Fl(E)_T$ with a tilde and classes on $\Fl(E)$ without. Recall the cohomology classes $\tilde{H}_{\ell+1, j}=-\pi^*c_1(L_j)$ on $\Fl(E)_T$, and  $H_{\ell + 1,j}=-\pi^*c_1(L_j)$ on $\Fl(E)$. 
For a fixed homology class $\tilde{\beta}$ on $\Fl(E)_T$ define $d_{\ell +1,j} =  \langle -\pi^*c_{1}(L_j), \tilde{\beta} \rangle$, and for a fixed homology class $\beta$ on $\Fl(E)$ define $d_{\ell +1,j} =  \langle -\pi^*c_{1}(L_j), \beta \rangle$. We use the indexing of the set $\{1, \dots, R\}$ defined in Section~\ref{notation}, and denote the components of a vector $\underline{d} \in \ZZ^R$ by $d_{i,j}$. Similarly, we denote components of a vector $\underline{d} \in \ZZ^\ell$ by $d_i$. 

In \cite{Oh2016}, the author proves that a certain generating function, the $I$-function of $\Fl(E)$, lies on the Lagrangian cone for $\Fl(E)$.
\begin{theorem}\label{ohI} 
Let $\tau \in H^{\bullet}(X)$, $t=\sum_{i}t_{i}c_1(S_i^\vee)$, and define the $I$-function of $\Fl(E)$ to be 
\begin{multline*}
	I_{\Fl(E)}(t, \tau,z) = \\
	e^{\frac{t}{z}} \sum_{\beta \in {\NE}(\Fl(E))} 
	Q^{\beta} e^{\langle \beta, t \rangle} \pi^{*}J_{X}^{\pi_*\beta}(\tau,z)  
	\sum_{\substack{\underline{d} \in \ZZ^R\colon \\ 
	\forall i \sum_j d_{i,j}=\langle \beta, c_1(S_i^\vee) \rangle}} 
	\prod_{i=1}^{ \ell} 
	\prod_{j=1}^{r_{i}}\prod_{j'=1}^{r_{i+1}} \frac{\prod_{m=-\infty}^{0}H_{i,j} - H_{i+1,j'} + mz}{\prod_{m=-\infty}^{d_{i,j} - d_{i+1,j'}}H_{i,j} - H_{i+1,j'} + mz}\\ 
	\times \prod_{i=1}^\ell \prod_{j \neq j'} 
	\frac{\prod_{m=-\infty}^{d_{i,j} - d_{i,j'}}H_{i,j} - H_{i,j'} + mz}{\prod_{m=-\infty}^{0} H_{i,j} -H_{i,j'} + mz}
\end{multline*}
Then $I_{\Fl(E)}(t,\tau,-z) \in \cL_{\Fl(E)}$ for all $t$ and $\tau$. 
\end{theorem}
In \cite{Brown2014}, the author proves an analogous result for the corresponding Abelian quotient $\Fl(E)_T$.

\begin{theorem}\label{brown2014gromov}
Let $\tau \in H^{\bullet}(X)$, $t=\sum_{i,j}t_{i,j}\tilde{H}_{i,j}$, and define the Brown $I$-function of $\Fl(E)_T$ to be
\begin{multline*}
I_{\Fl(E)_T}(t, \tau,z) = \\ 
e^{\frac{t}{z}} \sum_{\tilde{\beta} \in H_2\Fl(E)_T } Q^{\tilde{\beta}} e^{\langle \tilde{\beta}, t \rangle} \pi^{*}J_{X}^{\pi_* \tilde{\beta}}(\tau, z)
\prod_{i=1}^{\ell} \prod_{j=1}^{r_i}\prod_{j'=1}^{r_{i+1}}\frac{  \prod_{m=-\infty}^{0}\tilde{H}_{i,j} - \tilde{H}_{i+1,j'} + mz}
{\prod_{m=-\infty}^{\langle \tilde{\beta}, \tilde{H}_{i,j}  - \tilde{H}
_{i+1,j'} \rangle }\tilde{H}_{i,j} - \tilde{H}_{i+1,j'} + mz}
\end{multline*}
Then $I_{\Fl(E)_T}(t,\tau,-z) \in \cL_{\Fl(E)_T}$ for all $t$ and $\tau$.

\end{theorem}
\begin{remark}\label{novi}
We have chosen to state Theorem~\ref{brown2014gromov} in a different form than in Brown's original paper. The equivalence of the two versions follows from Lemma~\ref{effsummationrange} below. The classes $H_{i,j}$ here were denoted in~\cite{Brown2014} by $P_i$, and the classes $H_{i,j}-H_{i+1,j'}$ here were denoted there by~$U_k$. 
\end{remark}
\begin{lemma}\label{effsummationrange}
Writing $I_{\Fl(E)_T}=\sum_{\tilde{\beta}} I_{\Fl(E)_T}^{\tilde{\beta}} Q^{\tilde{\beta}}$,  any nonzero $I^{\tilde{\beta}}$ must have $\tilde{\beta} \in \NE(\Fl(E)_T)$. 
\end{lemma}
\begin{proof}
 To see this we temporarily adopt the notation of Brown and denote the torus invariant divisors by $U_k$, as in Lemma~\ref{torus_invariant_divisors}. 
Then $I_{\Fl(E)_T}$ takes the form 
\[
	I_{\Fl(E)_T}=\sum_{\substack{\tilde{\beta} \in H_2\Fl(E)_T \colon \\ \pi_*\tilde{\beta} \in \NE(X)}}(\dots)\prod_{k=1}^N \frac{\prod_{m=-\infty}^{0}U_k+mz}{\prod_{m=-\infty}^{\langle \tilde{\beta}, U_k \rangle}U_k+mz}
\]
Let $\alpha \subset \{1, \dots N\}$ be a subset of size $R$ which defines a section of the toric fibration as in Section~\ref{flag}. We have that 
\[
	s_\alpha^*I_{\Fl(E)_T}=(\dots)
	\prod_{k \in \alpha} \frac{\prod_{m=-\infty}^{0}(0) + mz}{\prod_{m=-\infty}^{\langle \tilde{\beta}, U_k \rangle}(0)+ mz}\prod_{k \notin \alpha} \frac{\prod_{m=-\infty}^{0}s_\alpha^*U_k+ mz}{\prod_{m=-\infty}^{\langle \tilde{\beta}, U_k \rangle}s_\alpha^*U_k+ mz}
\]
since $s_\alpha^*(U_k)=0$ if $k \in \alpha$. Therefore, if $\langle \tilde{\beta}, U_k \rangle < 0$ for some $k \in \alpha$, the numerator contains a term $(0)$ and vanishes.
We conclude that any $\tilde{\beta} \in H_2\Fl(E)_T$ which gives a nonzero contribution to $s_\alpha^*I_{\Fl(E)_T}$ must satisfy the conditions 
\[
	\pi_*\tilde{\beta} \in \NE(X), \langle \tilde{\beta}, U_k \rangle \geq 0 \, \forall k \in \alpha.
\]
The section $s_\alpha$ gives a splitting $H_2(\Fl(E)_T)=H_2(X)\oplus H_2(\Fl_T)$, via which we may write $\tilde{\beta}=s_{\alpha_*}D+\iota_*d$ where $\iota$ is the inclusion of a fibre. We have
\[
	\langle \tilde{\beta}, U_k \rangle= \langle  D, s_\alpha^*U_k  \rangle+ \langle d, \iota^*U_k \rangle=\langle d, \iota^*U_k \rangle
	 \geq 0
\]	
for all $k \in \alpha$. However, the cone in the secondary fan spanned by the line bundles $\iota^*U_k$ contains the ample cone of $\Fl_T$ (see Section \ref{flag}), so this implies $d \in \NE(\Fl_T)$. It follows that any $\tilde{\beta}$ which gives a nonzero contribution to $s_\alpha^*I_{\Fl(E)_T}$ is effective. 
We now use the Atiyah-Bott localization formula 
\[
	I_{\Fl(E)_T}=\sum_{\alpha} s_{\alpha_*}\left(\frac{s_\alpha^*I_{\Fl(E)_T}}{e^\alpha}\right), \quad \text{where} \; e^\alpha=\prod_{k \notin \alpha} s_\alpha^*U_k
\]
where $\alpha$ ranges over the torus fixed point sections of the fibration, to conclude that the same is true for $I_{\Fl(E)_T}$.
\end{proof}

\begin{lemma}\label{Brown I satisfies Divisor Equation}
Brown's $I$-function satisfies the Divisor Equation. That is,
$$z\nabla_{\rho}I_{\Fl(E)_T}^{\tilde{\beta}} = (\rho + \langle \rho,\tilde{\beta} \rangle z ) I_{\Fl(E)_T}^{\tilde{\beta}}$$ for any $\rho \in H^2(\Fl(E)_T)$.
\end{lemma}
\begin{proof}
Decompose $\rho=\rho_F + \pi^*\rho_B$ into fibre and base part.
	Basic differentiation and the divisor equation for $J_X$ show that
	$$ z\nabla_{\rho}I_{\Fl(E)_T}^{\tilde{\beta}} = \, \left(\rho_F + \langle \rho_F, \tilde{\beta} \rangle z+(\pi^*\rho_{B} + \langle \pi^*\rho_B, \tilde{\beta}\rangle z)\right) e^{t/z} e^{\langle \tilde{\beta}, t \rangle} \pi^{*}J_{X}^{\pi_{*}\tilde{\beta}}(\tau,z) \cdot \mathbf{H} $$	
	where $\mathbf{H}$ is a hypergeometric factor with no dependence on $t$ or $\tau$. The right-hand simplifies to 
	$$(\rho+ \langle \rho, \tilde{\beta} \rangle z)I_{\Fl(E)_T}^{\tilde{\beta}}$$ 
as required.
\end{proof}
\begin{lemma}\label{weylinvariant}
	If we restrict $t$ to lie in the Weyl-invariant locus $H^2(\Fl(E)_T)^W \subset H^2(\Fl(E)_T)$ then $(t,\tau) \mapsto I_{\Fl(E)_T}(t,\tau,z)$ takes values in $H^\bullet(\Fl(E)_T)^W$.
\end{lemma}
\begin{proof}
	This is immediate from the definition of $I_{\Fl(E)_T}(t,\tau,z)$, in Theorem~\ref{brown2014gromov}.
\end{proof}

\begin{proposition}\label{brownoh}
	Restrict $t$ to lie in the Weyl-invariant locus $H^2(\Fl(E)_T)^W \subset H^2(\Fl(E)_T)$ and consider the Brown $I$-function $(t, \tau) \mapsto I_{\Fl(E)_T}(t, \tau, z)$. The Givental--Martin modification $I_{\GM}(t, \tau)$ of this family is equal to Oh's $I$-function $I_{\Fl(E)}(t, \tau)$.
	\end{proposition}
	\begin{proof} 
		Lemma~\ref{weylinvariant} and Lemma~\ref{IGMexists} imply that the Givental--Martin modification $I_{\GM}(t, \tau)$ exists. We need to compute it. Note that the restrictions to the fibre of the classes $\tilde{H}_{i,j}$ form a basis for~$H^{2}(\Fl_T)$. Since the general fibre $\Fl_T$ of $\Fl(E)_T$ has vanishing first homology, the Leray--Hirsch theorem gives an identification $\QQ[H_2(\Fl(E)_T, \ZZ)]=\QQ[H_2(X,\ZZ)][q_{1,1}, \dots, q_{\ell, r_\ell}]$ via the map 
		\begin{equation}\label{noviab}
		Q^{\tilde{\beta}} \mapsto Q^{\pi_*{\tilde{\beta}}}\prod_{i,j}q_{i,j}^{\langle \tilde{H}_{i,j}, \tilde{\beta} \rangle}
		\end{equation}
		By Lemma~\ref{effsummationrange}, the summation range in the sum defining $I_{\Fl(E)_T}$ is contained in $\NE(\Fl(E)_T)$. We can therefore write the corresponding twisted $I$-function \eqref{general Weyl twist} as
		\begin{align*}
			I^{\tw}(t, \tau,z)= e^{\frac{t}{z}} \sum_{\substack{D \in \NE(X) \\ \underline{d} \in \ZZ^R}} Q^{D} \prod_{i,j}q_{i,j}^{d_{i,j}}e^{t \cdot \underline{d}} \pi^{*}J_{X}^{D}(\tau, z)
			\prod_{i=1}^{\ell} \prod_{j=1}^{r_i}\prod_{j'=1}^{r_{i+1}}\frac{  \prod_{m=-\infty}^{0}\tilde{H}_{i,j} - \tilde{H}_{i+1,j'} + mz}
			{\prod_{m=-\infty}^{d_{i,j}  - d_{i+1,j'}}\tilde{H}_{i,j} - \tilde{H}_{i+1,j'} + mz} \\
			\times \prod_{i=1}^{\ell} \prod_{j \neq j'} \frac{\prod_{m=-\infty}^{d_{i,j} - d_{i,j'}}\tilde{H}_{i,j} - \tilde{H}_{i,j'} + \lambda + mz}{\prod_{m=-\infty}^{0} \tilde{H}_{i,j} -\tilde{H}_{i,j'} + \lambda + mz}
		\end{align*}
		where the $t_{i,j} \in \CC$, $t = \sum_{i=1}^{\ell}\sum_{j=1}^{r_i} t_{i,j} \tilde{H}_{i,j}$, and $t \cdot \underline{d}=\sum_{i,j}t_{i,j}d_{i,j}$. For the Weyl modification factor we used the fact that the roots of $G$ are given by $\rho_{i,j} \rho_{i,j'}^{-1}$, where the character $\rho_{i,j}$ was defined in section \ref{flag}. By Lemma~\ref{effsummationrange} the effective summation range for the vector $\underline{d}$ here is contained in the set $S \subset \ZZ^R$ consisting of $\underline{d}$ such that $\langle \tilde{\beta}, \tilde{H}_{i,j} \rangle=d_{i,j}$ for some $\tilde{\beta} \in \NE(\Fl(E)_T)$. 

		We can identify the group ring $\QQ[H_2(\Fl(E))]$ with $\QQ[H_2(X,\ZZ)][q_{1}, \dots, q_{\ell}]$ via the map
		\begin{equation} \label{novinonab}
			Q^\beta \mapsto Q^{\pi_*\beta}\prod_{i}q_{i}^{\langle c_1(S_i^\vee), \beta \rangle}
		\end{equation}
		Via \eqref{noviab} and \eqref{novinonab} the map on Mori cones $\varrho: \NE(\Fl(E)_T) \rightarrow \NE(\Fl(E))$ becomes
		$$Q^D\prod_{i,j}q_{i,j}^{d_{i,j}} \mapsto Q^D\prod_{i}q_{i}^{\sum_j d_{i,j}}$$
		Restricting $t$ to the Weyl-invariant locus $H^2(\Fl(E)_T)^W$ corresponds to setting $t_{i,j}=t_i$ for all~$i$ and~$j$, which gives $e^{t \cdot \underline{d}}=e^{\sum_i t_i d_i}$ where $d_i=\sum_j d_{i,j}$. The identification $H^2(\Fl(E)_T)^W \cong H^2(\Fl(E))$ sends $\sum_{i,j} t_i \tilde{H}_{i,j}$ to $\sum_i t_i c_1(S_i^\vee)$, so projecting along \eqref{quotientH} and taking the limit as $\lambda = 0$ we obtain
		\begin{align*}
			e^{\frac{t}{z}} \sum_{\substack{D \in \NE(X)\\\underline{\delta} \in \ZZ^\ell}} Q^D\prod_{i}q_{i}^{\delta_i} e^{t \cdot \underline{\delta}} \pi^{*}J_{X}^D(\tau, z) \sum_{\substack{\underline{d} \in \ZZ^R\colon \\ 
			\forall i \sum_j d_{i,j}=\delta_i}}\prod_{i=1}^{\ell} \prod_{j=1}^{r_i}\prod_{j'=1}^{r_{i+1}}\frac{  \prod_{m=-\infty}^{0}{H}_{i,j} - {H}_{i+1,j'} + mz}{\prod_{m=-\infty}^{d_{i,j} - d_{i+1,j'}}{H}_{i,j} - {H}_{i+1,j'} + mz} \\
			\times \prod_{i=1}^{\ell} \prod_{j \neq j'} \frac{\prod_{m=-\infty}^{d_{i,j} - d_{i,j'}}H_{i,j} - H_{i,j'}  + mz}{\prod_{m=-\infty}^{0} H_{i,j} -H_{i,j'} + mz}
		\end{align*}
	where now $t=\sum_i t_i c_1(S_i^\vee)$. The effective summation range here is contained in $\NE(\Fl(E))$ by construction.
	Using \eqref{novinonab} again we may rewrite this as
	\begin{align*}
		e^{\frac{t}{z}} \sum_{\substack{\beta \in \NE(\Fl(E))}} Q^{\beta} e^{\langle \beta, t \rangle} \pi^{*}J_{X}^{\pi_*\beta}(\tau, z) \sum_{\substack{\underline{d} \in \ZZ^R\colon \\ \forall i \sum_j d_{i,j}=\langle \beta, c_1(S_i^\vee) \rangle}}\prod_{i=1}^{\ell} \prod_{j=1}^{r_i}\prod_{j'=1}^{r_{i+1}}\frac{  \prod_{m=-\infty}^{0}{H}_{i,j} - {H}_{i+1,j'} + mz}{\prod_{m=-\infty}^{d_{i,j} - d_{i+1,j'}}{H}_{i,j} - {H}_{i+1,j'} + mz} \\
		\times \prod_{i=1}^{\ell} \prod_{j \neq j'} \frac{\prod_{m=-\infty}^{d_{i,j} - d_{i,j'}}H_{i,j} - H_{i,j'}  + mz}{\prod_{m=-\infty}^{0} H_{i,j} -H_{i,j'} + mz}
	\end{align*}
	This is $I_{\Fl(E)}(t, \tau, z)$, as required. 
	\end{proof}

\begin{remark}\label{effectivesummationrange}
In view of \eqref{ampleconegen}, we see that the effective summation range in $I_{\Fl(E)}$ is contained in the subset of vectors satisfying
\[
	d_{i,j} \geq \min_{j'}d_{\ell+1,j'} \; \forall \, i, j
\]	
This will prove useful in calculations in Section \ref{examples}.
\end{remark}

\subsection{The Abelian/non-Abelian Correspondence with bundles}\label{AnAwithbundles}

We are now ready to prove Theorem~\ref{step two}. Recall from the Introduction that we have fixed a representation $\rho\colon  G \rightarrow \GL(V)$ where $G=\prod_i \GL_{r_i}(\CC)$, and that this determines vector bundles $V^G \to \Fl(E)$ and $V^T \to \Fl(E)_T$.  Since $T$ is Abelian, $V^T$ splits as a direct sum of line bundles $$V^T=F_1 \oplus \dots \oplus F_k$$
The Brown $I$-function gives a family 
\begin{align*}
	&(t,\tau) \mapsto I_{\Fl(E)_T}(t,\tau,{-z}) & \text{$t \in H^2(\Fl(E)_T)^W$, $\tau \in H^\bullet(X)$} 
	\intertext{of elements of $\cH_{\Fl(E)_T}$, and Theorem~\ref{brown2014gromov} shows that $I_{\Fl(E)_T}(t,\tau,{-z}) \in \cL_{\Fl(E)_T}$. Twisting by $(F, \mathbf{c})$ where $\mathbf{c}$ is the $\CC^\times$-equivariant Euler class with parameter $\mu$ gives a twisted $I$-function, as in Definition~\ref{twistedI}, which we denote by} 
	&(t, \tau) \mapsto I_{V^T_\mu}(t, \tau,{-z}) & \text{$t \in H^2(\Fl(E)_T)^W$, $\tau \in H^\bullet(X)$}
	\intertext{Applying Proposition~\ref{twisted=Deltad} shows that $I_{V^T_\mu}(t, \tau, {-z}) \in \cL_{V^T_\mu}$.  Twisting again, by $(\Phi, \mathbf{c'})$ where $\Phi \to \Fl(E)_T$ is the roots bundle from the Introduction and $\mathbf{c'}$ is the $\CC^\times$-equivariant Euler class with parameter $\lambda$ gives a twisted $I$-function, as in Definition~\ref{twistedI}, which we denote by} 
	& (t, \tau) \mapsto I_{\Phi_\lambda \oplus V^T_\mu}(t, \tau,{-z}) & \text{$t \in H^2(\Fl(E)_T)^W$, $\tau \in H^\bullet(X)$}
	\intertext{Applying Proposition~\ref{twisted=Deltad} again shows that $I_{\Phi_\lambda \oplus V^T_\mu}(t, \tau, {-z}) \in \cL_{\Phi_\lambda \oplus V^T_\mu}$. We now project along \eqref{quotientH} and take the non-equivariant limit $\lambda \to 0$, obtaining the Givental--Martin modification of~$I_{V^T_\mu}$. This is a family}
	& (t, \tau) \mapsto I_{\GM}(t, \tau, {-z}) & \text{$t \in H^2(\Fl(E)_T)^W$, $\tau \in H^\bullet(X)$}
\end{align*}
of elements of $\cH_{\Fl(E)}$. Explicitly:
\begin{definition}[which is a specialisation of Definition~\ref{IGM definition} to the situation at hand] \label{IGM definition special case} 
	\begin{align*}
		& I_{\GM}(t, \tau, z) = \\
		& e^{\frac{t}{z}} \sum_{\substack{\beta \in \NE(\Fl(E))}} Q^{\beta} e^{\langle \beta, t \rangle} \pi^{*}J_{X}^{\pi_*\beta}(\tau, z) 
		\sum_{\substack{\underline{d} \in \ZZ^R\colon \\ \forall i \sum_j d_{i,j}=\langle \beta, c_1(S_i^\vee) \rangle}}\prod_{i=1}^{\ell} \prod_{j=1}^{r_i}\prod_{j'=1}^{r_{i+1}}\frac{  \prod_{m=-\infty}^{0}{H}_{i,j} - {H}_{i+1,j'} + mz}{\prod_{m=-\infty}^{d_{i,j} - d_{i+1,j'}}{H}_{i,j} - {H}_{i+1,j'} + mz} \\
		& \qquad \qquad \qquad \qquad \qquad \qquad \qquad 
		\times \prod_{i=1}^{\ell} \prod_{j \neq j'} \frac{\prod_{m=-\infty}^{d_{i,j} - d_{i,j'}}H_{i,j} - H_{i,j'}  + mz}{\prod_{m=-\infty}^{0} H_{i,j} -H_{i,j'} + mz} 
		\prod_{s=1}^{k} \frac{\prod_{m=-\infty}^{f_s \cdot \underline{d}} f_s + \mu + mz}{\prod_{m=-\infty}^{0} f_s + \mu + mz}
	\end{align*}
	Here $J^D_X(\tau, z)$ is as in \eqref{JX by degrees}, $f_s \cdot \underline{d} = \sum_{i,j} f_{s,i,j} d_{i,j}$, and $f_s = \sum_{i,j} f_{s,i,j} H_{i,j}$, where $$c_1(F_s) = \sum_{i=1}^\ell \sum_{j=1}^{r_i}f_{s,i,j} \tilde{H}_{i,j}$$
\end{definition}

\noindent Lemma~\ref{IGMexists} shows that this expression is well-defined despite the presence of 
$$\omega = \textstyle \prod_i \prod_{j < j'} (H_{i,j} - H_{i,j'})$$ 
in the denominator. Corollary~\ref{IGMonLGM} shows that $I_{\GM}(t, \tau, {-z}) \in \cL_{\GM, V^T_\mu}$.  Note that $I_{\GM}(t, \tau)$ is \emph{not} the $V^G$-twist of Oh's $I$-function $I_{\Fl(E)}$. Indeed $V^G$ need not be a split bundle, so the twist may not even be defined. 

\begin{theorem} \label{IGM on twisted cone}
	Let $I_{\GM}$ be as in Definition~\ref{IGM definition special case}. Then:
	\begin{align*}
		I_{\GM}(t,\tau,-z) \in \cL_{V^G_\mu} &&
		\text{for all $t \in H^2(\Fl(E)_T)^W$, $\tau \in H^\bullet(X)$.}
	\end{align*}
\end{theorem}

\begin{proof}
	Before projecting and taking the non-equivariant limit, we have
	$$
	I_{\Phi_\lambda \oplus V^T_\mu} = \Delta_{V^T_\mu} \big( D_{V^T_\mu} \big( I_{\Phi_\lambda}\big)\big)
	$$
	by Proposition~\ref{twisted=Deltadeq}. Projecting along \eqref{quotientH} gives 
	$$
	p \circ I_{\Phi_\lambda \oplus V^T_\mu} = \Delta_{V^G_\mu} \big( D_{V^G_\mu} \big( p \circ I_{\Phi_\lambda}\big)\big)
	$$
	and taking the limit $\lambda \to 0$, which is well-defined by Lemma~\ref{IGMexists}, gives
	$$
	I_{\GM} = \Delta_{V^G_\mu} \big( D_{V^G_\mu} \big( I_{\Fl(E)}\big)\big)
	$$
	by Proposition~\ref{brownoh}. The result now follows from Proposition~\ref{twisted=Deltad}.
\end{proof}

Exactly the same argument proves:

\begin{corollary} \label{extra line bundle}
	Let $L \to X$ be a line bundle with first Chern class $\rho$, and define the vector bundle $F \to \Fl(E)$ to be $F = V^G \otimes \pi^* L$. Let $I_{\GM}$ be as in Definition~\ref{IGM definition special case}, except that the factor
	\begin{align*}
		\prod_{s=1}^{k} \frac{\prod_{m=-\infty}^{f_s \cdot \underline{d}} f_s + \mu + mz}{\prod_{m=-\infty}^{0} f_s + \mu + mz}
		&& \text{is replaced by} &&
		\prod_{s=1}^{k} \frac{\prod_{m=-\infty}^{f_s \cdot \underline{d} + \langle \rho, \pi_* \beta \rangle} f_s + \pi^* \rho + \mu + mz}{\prod_{m=-\infty}^{0} f_s + \pi^* \rho + \mu + mz}
	\end{align*}
	Then:
	\begin{align*}
		I_{\GM}(t,\tau,-z) \in \cL_{F_\mu} &&
		\text{for all $t \in H^2(\Fl(E)_T)^W$, $\tau \in H^\bullet(X)$.}
	\end{align*}
\end{corollary}

The following Corollary gives a closed-form expression for genus-zero Gromov--Witten invariants of the zero locus of a generic section $Z$ of $F$ in terms of invariants of $X$.

\begin{corollary} \label{I=J}
	With notation as in Corollary~\ref{extra line bundle}, let $Z$ be the zero locus of a generic section of $F \rightarrow \Fl(E)$. Suppose that ${-K_Z}$ is the restriction of an ample class on $\Fl(E)$ and that $\tau \in H^2(X)$. Then 
	$$J_{F_\mu}(t+\tau, z)=e^{-C(t)/z}I_{\GM}(t,\tau, z)$$
	where
	$$ C(t)=\sum_{\beta} n_\beta Q^\beta e^{\langle \beta, t \rangle} $$
	for some constants $n_\beta \in \QQ$ and the sum runs over the finite set 
	$$S=\{ \beta \in \NE(\Fl(E)) : \langle {-K}_{\Fl(E)} - c_1(F), \beta \rangle = 1\}$$ 
	If $Z$ is of Fano index two or more then this set is empty and $C(t) \equiv 0$.  Regardless, if the vector bundle $F$ is convex then the non-equivariant limit $\mu \to 0$ of $J_{F_\mu}$ exists and 
	$$
	J_Z\big(i^* t + i^* \tau, z \big) = i^* J_{F_0}(t+\tau, z)
	$$
	where $i \colon Z \to \Fl(E)$ is the inclusion map.
\end{corollary}
\begin{proof}[Proof of Corollary~\ref{I=J}]
	The statement about Fano index two or more follows immediately from the Adjunction Formula 
	$$ K_Z=\big(K_{\Fl(E)}+c_1(F)\big)\big|_Z$$
	We need to show that 
	\begin{equation} \label{IGM asymptotics}
		I_{\GM}(t, \tau, z) = z + t + \tau + C(t) + O(z^{-1})
	\end{equation}
	Everything else then follows from the characterisation of the twisted $J$-function just below Definition~\ref{twisted J}, the String Equation
	\begin{align*}
		J_{F_\mu}(\tau + a, z) = e^{a/z} J_{F_\mu}(\tau,z) &&
		a \in H^0(\Fl(E))
	\end{align*}
	and~\cite{Coates2014}. To establish \eqref{IGM asymptotics}, it will be convenient to set $\deg(z) = \deg(\mu) = 1$, $\deg(\phi)=k$ for $\phi \in H^{2k}(\Fl(E))$, and $\deg(Q^\beta)=\langle -K_X, \beta \rangle$ if $\beta \in H_2(X)$. The degree axiom for Gromov--Witten invariants then shows that $J_X^{\pi_*\beta}$ is  homogeneous of degree $\langle K_X, \pi_*\beta \rangle +1$. Write
	$$I_{\GM}(t, \tau,z) = e^{\frac{t}{z}} \sum_{\beta \in {\NE}(\Fl(E))} 
	Q^{\beta} e^{\langle \beta, t \rangle} \pi^{*}J_{X}^{\pi_*\beta}(\tau, z)  \times {I}_{\beta}(z) \times M_\beta(z)
	$$
	where
	$$ M_\beta(z) = \prod_{s=1}^{k} \frac{\prod_{m=-\infty}^{f_s \cdot \underline{d} + \langle \rho, \pi_* \beta \rangle} f_s + \pi^* \rho + \mu + mz}{\prod_{m=-\infty}^{0} f_s + \pi^* \rho + \mu + mz}
	$$
	A straightforward calculation shows that 
	\begin{align*}
	{I}_{\beta}(z)&=z^{\langle K_{\Fl(E)}-\pi^*K_X, \beta \rangle}{i}_{\beta}(z)\\
	M_\beta(z)&=z^{\langle c_1(F), \beta \rangle}m_\beta(z)
	\end{align*}
	where $i_\beta(z), m_\beta(z) \in \cH_{\Fl(E)}$ are homogeneous of degree $0$.
	It follows that $\pi^{*}J_{X}^{\pi_*\beta}(\tau, z)  \times {I}_{\beta}(z) \times M_\beta(z)$ is homogeneous of degree $\langle K_{\Fl(E)}+c_1(F), \beta \rangle+1$ which is nonpositive for $\beta \neq 0$ by the assumptions on $-K_Z$. Since $\tau \in H^2(X)$, any negative contribution to the homogenous degree must come from a negative power of $z$, so that $\pi^{*}J_{X}^{\pi_*\beta}(\tau, z)  \times {I}_{\beta}(z) \times M_\beta(z)$ is $O(z^{-1})$, unless $\beta=0$ or $\beta \in S$. In the latter case, the expression has homogeneous degree $0$ and is therefore of the form $c_0+\tfrac{c_1}{z}+O(z^{-2})$ with $c_i$ independent of $z$ and of degree $i$. Relabeling $n_\beta=c_0$ and
expanding $I_{\GM}$ in powers of $z$, we obtain 
	\begin{multline*}
	I_{\GM}(t, \tau,z) =
	\big(1+t z^{-1}+O(z^{-2})\big) 
\big(\pi^*J_X^0 \times I_{0} \times M_0 +\Big(
	\sum_{\beta \in S}
	n_\beta Q^{\beta} e^{\langle \beta, t \rangle}  +O(z^{-1})\Big) +\sum_{0 \neq \beta \notin S} O(z^{-1})\big)\\
	=(z+\tau+t+C(t)+O(z^{-1}))
	\end{multline*}
	where $C(t)$ is as claimed. This proves \eqref{IGM asymptotics}, and the result follows. 
\end{proof}

We restate Corollary~\ref{I=J} in the case where the flag bundle is a Grassmann bundle, i.e $\ell=1$, relabelling $H_{1,j}=H_j$, $d_{1,j} = d_j$ and $r_1=r$. The rest of the notation here is as in \S\ref{notation}.

\begin{corollary}\label{explicit Gr} 
Let $V^G\rightarrow \Gr(r, E)$ be a vector bundle induced by a representation of $G$, let $L \to X$ be a line bundle with first Chern class $\rho$, and let $F = V^G \otimes \pi^* L$. Let $Z$ be the zero locus of a generic section of $F$. Suppose that $F$ is convex, that $-K_{\Gr(E,r)} - c_1(F)$ is ample, and that $\tau \in H^2(\Gr(r,E))$. Then the non-equivariant limit $\mu \to 0$ of the twisted $J$-function $J_{F_\mu}$ exists and satisfies
$$ J_Z\big(i^* t + i^*\tau, z \big) = i^* J_{F_0}(t+\tau, z) $$
where $i \colon Z \to \Gr(r, E)$ is the inclusion map. Furthermore
\begin{multline} \label{explicit J for Gr}
	J_{F_0}(t+\tau, z) = e^{\frac{t - C(t)}{z}} \sum_{\substack{\beta \in \NE(\Gr(r, E))}} Q^{\beta} e^{\langle \beta, t \rangle} \pi^{*}J_{X}^{\pi_*\beta}(\tau, z) \\
	\sum_{\substack{\underline{d} \in \ZZ^r\colon \\ d_1 + \cdots + d_r = \langle \beta, c_1(S^\vee) \rangle}} (-1)^{\epsilon(\underline{d})} 
	\prod_{i=1}^{r}\prod_{j=1}^{n}\frac{  \prod_{m=-\infty}^{0}{H}_i + \pi^* c_1(L_j) + mz}{\prod_{m=-\infty}^{d_i + \langle \pi_* \beta, c_1(L_j) \rangle} H_i + \pi^* c_1(L_j) + mz} \\
	\\ 
	\times \prod_{i < j} \frac{H_i - H_j  + (d_i - d_j)z}{H_i -H_j} 
	\times \prod_{s=1}^{k} \prod_{m=1}^{f_s \cdot \underline{d} + \langle \rho, \pi_* \beta \rangle} \big( f_s + \pi^* \rho + mz \big)
\end{multline}
Here the Abelianised bundle $V^T$ splits as a direct sum of line bundles $F_1 \oplus \cdots \oplus F_k$ with first Chern classes that we write as $c_1(F_s) = \sum_{i=1}^{r}f_{s,i} \tilde{H}_i$, $J^D_X(\tau,z )$ is as in \eqref{JX by degrees}, $\epsilon(\underline{d}) = \sum_{i<j} d_i - d_j$, $f_s \cdot \underline{d} = \sum_{i} f_{s,i} d_i$, $f_s = \sum_i f_{s,i} H_i$, and $C(t) \in H^0(\Gr(r, E), \Lambda)$ is the unique expression such that the right-hand side of \eqref{explicit J for Gr} has the form $z + t + \tau + O(z^{-1})$. 
\end{corollary}

\begin{remark} \label{effectivesummationrange Gr}
	For a more explicit formula for $C(t)$, see Corollary~\ref{I=J}; in particular if $Z$ has Fano index two or greater then $C(t) \equiv 0$. By Remark~\ref{effectivesummationrange} the summand in \eqref{explicit J for Gr} is zero unless for each $i$ there exists a $j$ such that $d_i + \langle \pi_* \beta, c_1(L_j) \rangle \geq 0$
\end{remark}

\begin{proof}[Proof of Corollary~\ref{explicit Gr}]
	We cancelled terms in the Weyl modification factor, as in the proof of Lemma~\ref{IGMexists}, and took the non-equivariant limit $\mu \to 0$.
\end{proof}

\begin{remark}
	The relationship between $I$-functions (or generating functions for genus-zero quasimap invariants) and $J$-functions (which are generating functions for genus-zero Gromov--Witten invariants) is particularly simple in the Fano case~\cite{Givental1996toric}~\cite[\S1.4]{CFK2014}, and for the same reason Corollary~\ref{I=J} holds without the restriction $\tau \in H^2(X)$ if $Z \to X$ is relatively Fano\footnote{That is, if the relative anticanonical bundle ${-K}_{Z/X}$ is ample.}. This never happens for blow-ups $\tilde{X} \to X$, however, and it is hard to construct examples where $Z \to X$ is relatively Fano and the rest of the conditions of Corollary~\ref{I=J} hold. We do not know of any such examples.
\end{remark}

\begin{remark} \label{what can we compute}
	Corollary~\ref{I=J} gives a closed-form expression for the small $J$-function of $Z$ -- or, equivalently, for one-point gravitional descendant invariants of $Z$ -- in the case where $Z$ is Fano. But in general (that is, without the Fano condition on $Z$) one can use Birkhoff factorization, as in~\cite{CoatesGivental2007, CFK2014} and~\cite[\S3.8]{CCIT2019}, to compute any twisted genus-zero gravitional descendant invariant of $\Fl(E)$ in terms of genus-zero descendant invariants of $X$. The twisting here is with respect to the $\CC^\times$-equivariant Euler class and the vector bundle $F$. Thus Corollary~\ref{I=J} determines the Lagrangian submanifold $\cL_{F_\mu}$ that encodes twisted Gromov--Witten invariants. Applying~\cite[Theorem~1.1]{Coates2014}, we see that Corollary~\ref{I=J} together with Birkhoff factorization allows us to compute any genus-zero Gromov--Witten invariant of the zero locus $Z$ of the form
	\begin{equation} \label{what can we compute GW}
		\langle \theta_1 \psi^{i_1}, \dots, \theta_n \psi^{i_n} \rangle_{0,n,d} 
	\end{equation}
	where all but one of the cohomology classes $\theta_i$ lie in $\image(i^*) \subset H^\bullet(Z)$ and the remaining $\theta_i$ is an arbitrary element of $H^\bullet(Z)$.  Here $i \colon Z \to \Fl(E)$ is the inclusion map.
\end{remark}

\begin{remark} \label{what can we compute blow up}
	Applying Remark~\ref{what can we compute} to the blow-up $\tilde{X} \to X$ considered in the introduction, we see that Corollary~\ref{I=J} together with Birkhoff factorization allows us to compute arbitrary invariants of $\tilde{X}$ of the form \eqref{what can we compute GW} in terms of genus-zero gravitional descendants of $X$. In this case $\image(i^*) \subset H^\bullet(\tilde{X})$ contains all classes from $H^\bullet(X)$ and also the class of the exceptional divisor.
\end{remark}
\section{The Main Geometric Construction} \label{geometric}

\subsection{Main Geometric Construction}
Let $F$ be a locally free sheaf on a variety $X$. We denote by $F(x)$ its fibre over $x$, a vector space over the residue field $\kappa(x)$. A morphism $\varphi$ of locally free sheaves induces a linear map on fibres, denoted by $\varphi(x)$. We make the following definition:
\begin{definition}
Let $\varphi \colon E^m \rightarrow F^n$ a morphism of locally free sheaves of rank $m$ and $n$ respectively. 
	The $k$-th degeneracy locus is the subvariety of $X$ defined by $$D_k(\varphi)=\big \{ x \in X\colon  \rk \, 
\varphi(x)  \leq k \big\}$$
Note that $D_k(\varphi)=X$ if $k \geq \min\{m,n\}$; if $k=\min\{m,n\}-1$ we simply call $D_k(\varphi)$ the degeneracy locus of $\varphi$.
\end{definition}
We have the following results:
\begin{itemize}
	\item Scheme-theoretically, $D_k(\varphi)$ may be defined as the zero locus of the section $\wedge^k\varphi$; this shows that locally the ideal of $D_k(\varphi)$ is defined by the $(k+1) \times (k+1)$-minors of $\varphi$.
	\item If $E^\vee \otimes F$ is globally generated, then $D_k(\varphi)$ of a generic $\varphi$ is either empty or has expected codimension $(m-k)(n-k)$, and the singular locus of $D_k(\varphi)$ is contained in $D_{k-1}(\varphi)$. In particular, if $\varphi$ is generic and $\dim X < (m-k+1)(n-k+1)$, then $D_k(\varphi)$ is smooth \cite[Theorem 2.8]{Ottaviani1995}.
	\item We may freely assume that $m \geq n$ in what follows, since we can always replace  $\varphi$ with its dual map whose degeneracy locus is the same. 
\end{itemize}

\begin{proposition}\label{geometricconstruction}
Let X be a smooth variety, and $\varphi\colon E^m \rightarrow F^n$ a generic morphism of locally free sheaves on $X$. Suppose that $m \geq n$ and write $r=m-n$. Let $Y=D_{n-1}(\varphi)$ be the degeneracy locus of $\varphi$, and assume that $\varphi$ has generically full rank, that $Y$ has the expected codimension $m-n+1$ and that $Y$ is smooth. 
Let  $\pi\colon \Gr(r,E) \rightarrow X$ be the Grassmann bundle of $E$ on $X$.
Then the blow-up $Bl_Y(X)$ of $X$ along $Y$ is a subvariety of $\Gr(r,E)$, cut out as the zero locus of the regular section $s \in \Gamma(\Hom(S, \pi^*F))$ defined by the composition $$S \hookrightarrow \pi^*E \xrightarrow{\pi^*\varphi} \pi^*F $$
where the first map is the canonical inclusion.
\end{proposition}
\begin{proof}
We write points in $\Gr(r,E)$ as $(p, V)$, where $p \in X$ and $V$ is a $r$-dimensional subspace of the fibre $E(x)$. At $(p, V)$, the section $s$ is given by the composition 
$$V \hookrightarrow E(x) \xrightarrow{\varphi(x)} F(x)$$
so $s$ vanishes at $(p, V)$ if and only if $V \subset \ker\varphi(x)$.

The statement is local on $X$, so fix a point $P \in X$ and a Zariski open neighbourhood $U=\Spec(A)$  with trivialisations $E |_U=A^m, F |_U=A^n $. We will show that the equations of $Z(s) \cap U$ and $Bl_{U \cap Y}U$ agree. Under these identifications
$\varphi$ is given by a $n \times m$ matrix with entries in $A$. Since $\varphi$ has generically maximal rank and $Y$ is nonsingular, after performing row and column operations and shrinking $U$ if necessary, we may assume that $\varphi$ is given by the matrix 
$$\begin{pmatrix}
x_0&\dots &x_{r}&0&0&\dots&0\\
0& \dots &0&1&0&\dots&0\\
0& \dots &0&0&1&\dots&0\\
\vdots&\vdots&\vdots&\vdots&\vdots&\vdots&\vdots\\
0&\dots&0&0&\dots&0&1\\
\end{pmatrix}$$
Note that the ideal of the minors of this matrix is just $I=(x_0, \dots x_{r})$ and that $x_0, \dots, x_r$ form part of a regular system of parameters around $P$,  so we may assume that $n=1, m=r+1$.
Writing $y_i$ for the basis of sections of $S^\vee$ on $\Gr(r,A^{r+1})$, we see that $Z(s)$ is given by the equation
\begin{align*}
x_0y_0+\dots+x_ry_r&=0
\end{align*}
Under the Pl\"ucker isomorphism $$\Gr(r,A^{r+1}) \rightarrow \PP(\wedge^{r}A^{r+1})\cong U \times \PP^r_{y_0,\dots, y_r}$$
$Z(s)$ maps to the variety cut out by the minors of the matrix 
$$\begin{pmatrix}
x_0&\dots& x_{r}\\
{y}_0& \dots &{y}_{r}
\end{pmatrix}$$
i.e the blowup of $Y \cap U$ in $U$. 
\end{proof}

\section{Examples} \label{examples}

We close by presenting three example computations that use Theorems~\ref{step one} and~\ref{step two}, calculating genus-zero Gromov--Witten invariants of blow-ups of projective spaces in various high-codimension complete intersections. Recall, as we will need it below, that if $E \to X$ is a vector bundle of rank $n$ then the anticanonical divisor of $\Gr(r,E)$ is  
\begin{equation}
	\label{-K}
	{-K}_{\Gr(r,E)}=\pi^*\left(-K_X+r(\det E)\right) +n(\det S^\vee) 
\end{equation}
where $S \to \Gr(r,E)$ is the tautological subbundle. Recall too that the \emph{regularised quantum period} of a Fano manifold $Z$ is the generating function
$$
\widehat{G}_Z(x) = 1 + \sum_{d = 2}^\infty d! c_d x^d
$$
for genus-zero Gromov--Witten invariants of $Z$, where
\begin{align*}
	c_d = \sum_{\beta} \langle \theta \psi_1^{d-2} \rangle_{0,1,\beta}
	&& \text{for $\theta \in H^{\text{top}}(Z)$ the class of a volume form}
\end{align*}
and the sum runs over effective classes $\beta$ such that $\langle \beta, {-K}_Z\rangle = d$.

\begin{example}
We will compute the regularised quantum period of $\tilde{X}=\Bl_Y\PP^4$ where $Y$ is a plane conic. Consider the situation as in \S\ref{notation} with:
\begin{itemize}
	\item $X=\PP^4$ 
	\item $E= \cO \oplus \cO \oplus \cO(-1)$
	\item $G = \GL_{2}(\CC)$,  $T= (\CC^\times)^2 \subset G$
\end{itemize}
Then $A \GIT G$ is $\Gr(2, E)$, and $A \GIT T$ is the $\PP^2 \times \PP^2$-bundle $\PP(E) \times_{\PP^4} \PP(E) \to \PP^4$. By Proposition \ref{geometricconstruction} the zero locus $\tilde{X}$ of a section of $S^\vee \otimes \pi^*(\OO(1))$ on $\Gr(2,E)$ is the blowup of $\PP^4$ along the complete intersection of two hyperplanes and a quadric. We identify the group ring $\QQ[H_2(A \GIT T,\ZZ)]$ with $\QQ[Q,Q_1,Q_2]$, where $Q$ corresponds to the pullback of the hyperplane class of $\PP^4$ and $Q_i$ corresponds to $\tilde{H}_i$. Similarly, we identify $\QQ[H_2(A \GIT G,\ZZ)]$ with $\QQ[Q, q]$, where again $Q$ corresponds to the pullback of the hyperplane class of $\PP^4$ and $q$ corresponds to the first Chern class of $S^\vee$.

We will need Givental's formula~\cite{Givental1996equivariant} for the $J$-function of $\PP^4$:
\begin{align*}
	J_{\PP^4}(\tau, z)=z e^{\tau/z}\sum_{D = 0}^\infty \frac{Q^D e^{D\tau}}{\prod_{m=1}^D (H+mz)^5} && \tau \in H^2(\PP^4)
\end{align*}
In the notation of \S\ref{notation}, we have $\ell=1$, $r_\ell=r_1=2$, $r_{\ell+1}=3$. We relabel $\tilde{H}_{\ell,j}=\tilde{H}_j$ and $d_{\ell, j}=d_j$. We have that $\tilde{H}_{\ell+1, 1}=\tilde{H}_{\ell+1, 2}=0$, $\tilde{H}_{\ell+1, 3}=\pi^*H$ and $d_{\ell+1, 1}=d_{\ell+1, 2}=0$, $d_{\ell+1, 3}=D$. Write $F = S^\vee \otimes \pi^* \cO(1)$. Corollary~\ref{explicit Gr} and Remark~\ref{effectivesummationrange Gr} give
\begin{multline*}
J_{F_0}(t, \tau,z) = z e^{\frac{t + \tau}{z}}\sum_{D=0}^\infty \sum_{d_1=0}^\infty \sum_{d_2=0}^\infty \frac{(-1)^{d_1-d_2} Q^{D} q^{d_1+d_2} e^{D\tau}e^{(d_1 +d_2)t} \prod_{i=1}^{2} \prod_{m=1}^{d_i + D} (H_i + H + mz)}{\prod_{m=1}^{D} (H+ mz)^5 \prod_{m=1}^{d_1}(H_1+mz)^2\prod_{m=1}^{d_2}(H_2+mz)^2} \\
\times \prod_{i=1}^2\frac{\prod_{m=-\infty}^{0} (H_i -H + mz)}{\prod_{m=-\infty}^{d_i-D}(H_i -H + mz)}
 \frac{(H_1 - H_2 + z(d_1 - d_2))}{H_1 - H_2}
\end{multline*}
To obtain the quantum period we need to calculate the anticanonical bundle of $\tilde{X}$.
Equation \eqref{-K} and the adjunction formula give
$$-K_{\widetilde{X}}=3H+3\det S^\vee-(2H+\det S^\vee)=H+2 \det S^\vee.$$
To extract the quantum period from the non-equivariant limit $J_{F_0}$ of the twisted $J$-function, we take the component along the unit class $1 \in H^\bullet(A \GIT G; \QQ)$, set $z=1$, and set $Q^\beta=x^{\langle \beta, -K_{\tilde{X}} \rangle}$. That is, we set $\lambda = 0$, $t=0$, $\tau=0$, $z=1$, $q=x^2$, $Q=x$, and take the component along the unit class, obtaining 
\begin{multline*}
	G_{\tilde{X}}(x)= \sum_{n=0}^\infty \sum_{l=n+1}^\infty \sum_{m=l}^\infty \textstyle (-1)^{l+m-1}x^{l+2m+2n} 
	\frac{(l+n)!(l+m)!(l-n-1)!}{(l!)^5(m!)^2(n!)^2(n-l)!}(n-m)\\
	+ \sum_{l=0}^\infty \sum_{m=l}^\infty \sum_{n=l}^\infty \textstyle (-1)^{m+n}x^{l+2m+2n} 
	\frac{(l+n)!(l+m)!}{(l!)^5(m!)^2(n!)^2(n-l)!(m-l)!}
	\Big(1+(n-m)(-2H_{n}+H_{l+n}-H_{n-l}) \Big) 
\end{multline*}
Thus the first few terms of the regularized quantum period are:
\begin{multline*}
	\widehat{G}_{\tilde{X}}(x)=1+12x^3+120x^5+540x^6+20160x^8+33600x^9+113400x^{10} \\ +2772000x^{11}+2425500x^{12}+\cdots
\end{multline*}
This strongly suggests that $\tilde{X}$ coincides with the quiver flag zero locus with ID 15 in~\cite{Kalashnikov2019}, although this is not obvious from the constructions.
\end{example}

\begin{example}
We will compute the regularised quantum period of $\tilde{X}=\Bl_Y\PP^6$, where $Y$ is a 3-fold given by the intersection of a hyperplane and two quadric hypersurfaces. Consider the situation as in \S\ref{notation} with:
\begin{itemize}
	\item $X=\PP^6$ 
	\item $E= \cO \oplus \cO \oplus \cO(1)$
	\item $G = \GL_{2}(\CC)$,  $T= (\CC^\times)^2 \subset G$
\end{itemize}
Then $A \GIT G$ is $\Gr(2, E)$, and $A \GIT T$ is the $\PP^2 \times \PP^2$-bundle $\PP(E) \times_{\PP^6} \PP(E) \to \PP^6$. By Proposition \ref{geometricconstruction} the zero locus $\tilde{X}$ of a section of $S^\vee \otimes \pi^*(\OO(2))$ on $\Gr(2,E)$ is the blowup of $\PP^6$ along the complete intersection of a hyperplane and two quadrics. We identify the group ring $\QQ[H_2(A \GIT T,\ZZ)]$ here with $\QQ[Q,Q_1,Q_2]$, where $Q$ corresponds to the pullback of the hyperplane class of $\PP^6$ and $Q_i$ corresponds to $\tilde{H}_i$. Similarly, we identify $\QQ[H_2(A \GIT G,\ZZ)]$ with $\QQ[Q, q]$, where again $Q$ corresponds to the pullback of the hyperplane class of $\PP^6$ and $q$ corresponds to the first Chern class of $S^\vee$.

The $J$-function of $\PP^6$ is~\cite{Givental1996equivariant}:
\begin{align*}
	J_{\PP^6}(\tau, z)=z e^{\tau/z}\sum_{D = 0}^\infty \frac{Q^D e^{D\tau}}{\prod_{m=1}^D (H+mz)^7} && \tau \in H^2(\PP^6)
\end{align*}
In the notation of \S\ref{notation}, we have $\ell=1$, $r_\ell=r_1=2$, $r_{\ell+1}=3$. We relabel $\tilde{H}_{\ell,j}=\tilde{H}_j$ and $d_{\ell, j}=d_j$. We have that $\tilde{H}_{\ell+1, 1}=\tilde{H}_{\ell+1, 2}=0$, $\tilde{H}_{\ell+1, 3}=- \pi^*H$ and $d_{\ell+1, 1}=d_{\ell+1, 2}=0$, $d_{\ell+1, 3}=-D$. Write $F = S^\vee \otimes \pi^* \cO(2)$. Corollary~\ref{explicit Gr} and Remark~\ref{effectivesummationrange Gr} give
\begin{multline*}
J_{F_0}(t, \tau, z) = z e^{\frac{t+\tau}{z}} \sum_{D=0}^\infty \sum_{d_1 = -D}^\infty \sum_{d_2= -D}^\infty \frac{Q^D q^{d_1+d_2} e^{D\tau}e^{(d_1+d_2)t}}{\prod_{m=1}^D(H+mz)^7} \prod_{i=1}^2
\frac{\prod_{m=-\infty}^{0}(H_i+mz)^2}
{\prod_{m=-\infty}^{d_i}(H_i+mz)^2}\\
\times
\prod_{i=1}^2
\frac{\prod_{m=1}^{d_i+2D}(H_i+2H+mz)}{\prod_{m=1}^{d_i + D} (H_i + H + mz)}
(-1)^{d_1-d_2}\frac{(H_1-H_2+z(d_1-d_2))}{H_1-H_2}
\end{multline*}
Again we will need the anticanonical bundle of $\tilde{X}$, which by \eqref{-K} and the adjunction formula is 
$$-K_{\widetilde{X}}=9H + 3 \det(S^*)-(4H+\det(S^*))=5H + 2\det(S^*).$$
To extract the quantum period from $J_{F_0}$, we take the component along the unit class $1 \in H^\bullet(A \GIT G; \QQ)$, set $z=1$, and set $Q^\beta=x^{\langle \beta, -K_{\tilde{X}} \rangle}$. 
That is, we set $\lambda = 0$, $t=0$, $\tau=0$, $z=1$, $q=x^2$, $Q=x^5$, and take the component along the unit class, obtaining 
 \begin{multline*}
 G_{\tilde{X}}(x)=\sum_{D=0}^\infty \sum_{d_1=0}^\infty \sum_{d_2=0}^\infty (-1)^{d_1 + d_2}x^{5D+2d_1+2d_2} 
 \frac{(d_1 + 2D)! (d_2 + 2D)!}{(D!)^7(d_1!)^2(d_2!)^2(d_1 +D)!(d_2 + D)!} \\
 \times \Big(1+(d_1-d_2)(-2H_{d_1}+H_{d_1+2D}-H_{d_1+D}) \Big)
 \end{multline*}
 The first few terms of the regularized quantum period are:
 $$\widehat{G}_{\tilde{X}}(x) = 1+ 480x^5 + 5040 x^7 + 4082400 x^{10} + 119750400 x^{12} +  681080400 x^{14} +  \cdots$$
\end{example}

\begin{example}
We will compute the regularised quantum period of $\tilde{X}=\Bl_Y\PP^6$, where $Y$ is a quadric surface given by the intersection of 3 generic hyperplanes and a quadric hypersurface. Consider the situation as in \S\ref{notation} with:
\begin{itemize}
	\item $X=\PP^6$ 
	\item $E= \cO \oplus \cO \oplus \cO \oplus \cO(2)$
	\item $G = \GL_{3}(\CC)$,  $T= (\CC^\times)^3 \subset G$
\end{itemize}
Then $A \GIT G$ is $\Gr(3, E)$, and $A \GIT T$ is $\PP(E) \times_{\PP^6} \PP(E) \times_{\PP^6} \PP(E) \to \PP^6$. By Proposition \ref{geometricconstruction} the zero locus $\tilde{X}$ of a section of $S^\vee \otimes \pi^*(\OO(1))$ on $\Gr(3,E)$ is the blowup of $\PP^6$ along the complete intersection of three hyperplanes and a quadric. We identify the group ring $\QQ[H_2(A \GIT T,\ZZ)]$ with $\QQ[Q,Q_1,Q_2,Q_3]$, where $Q$ corresponds to the pullback of the hyperplane class of $\PP^6$ and $Q_i$ corresponds to $\tilde{H}_i$. Similarly, we identify $\QQ[H_2(A \GIT G,\ZZ)]$ with $\QQ[Q, q]$, where again $Q$ corresponds to the pullback of the hyperplane class of $\PP^6$ and $q$ corresponds the first Chern class of $S^\vee$. 

In the notation of \S\ref{notation}, we have $\ell=1, r_\ell=r_1=3, r_{\ell+1}=4$. We relabel $\tilde{H}_{\ell,j}=\tilde{H}_j$ and $d_{\ell, j}=d_j$. We have that $\tilde{H}_{\ell+1, 1}=\tilde{H}_{\ell+1, 2} = \tilde{H}_{\ell + 1,3}=0$, $\tilde{H}_{\ell+1, 4}=- \pi^*2H$ and $d_{\ell+1, 1}=d_{\ell+1, 2}=d_{\ell + 1,3} = 0$, $d_{\ell+1, 4}=-2D$. 
Write $F = S^\vee \otimes \pi^* \cO(1)$. Corollary~\ref{explicit Gr} and Remark~\ref{effectivesummationrange Gr} give
\begin{multline*}
J^{F_0}(t, \tau, z) = z e^{\frac{t+\tau}{z}} 
\sum_{D = 0}^\infty \sum_{d_1 = -2D}^\infty \sum_{d_2 = -2D}^\infty \sum_{d_3 = -2D}^\infty \frac{Q^{D}q^{d_1+d_2+d_3} e^{D\tau}e^{(d_1+d_2+d_3) t}}{\prod_{m=1}^D(H+mz)^7}\\
\times \prod_{i=1}^3
\frac{\prod_{m=-\infty}^{0}(H_i+mz)^3}
{\prod_{m=-\infty}^{d_i}(H_i+mz)^3}
\prod_{i=1}^3
\frac{1}
{\prod_{m=1}^{d_i+2D} (H_i + 2H+ mz)}
\prod_{i=1}^3
\frac{\prod_{m=-\infty}^{d_i+D}(H_i+H+mz)}{\prod_{m=-\infty}^{0}(H_i+H+mz)} \\
\times \frac{(H_1-H_2+z(d_1-d_2))}{H_1-H_2}
\frac{(H_1-H_3+z(d_1-d_3))}{H_1-H_3}\frac{(H_2-H_3+z(d_2-d_3))}{H_2-H_3}
\end{multline*}
Arguing as before,
$$-K_{\widetilde{X}}=11H + 4 \det(S^*)-(3H+\det(S^*))=8H + 3\det(S^*).$$
To extract the quantum period from $J_{F_0}$, we set $\lambda = 0$, $t=0$, $\tau=0$, $z=1$, $q=x^3$, $Q=x^8$, and take the component along the unit class. The first few terms of the regularised quantum period are:
\begin{multline*}
	\widehat{G}_{\tilde{X}}(x) = 1+ 108x^3 + 17820 x^6 + 5040 x^{8} + 5473440 x^{9} +  56364000 x^{11} + 1766526300 x^{12} \\ + 117076459500 x^{14} + 672012949608 x^{15} + \cdots
\end{multline*}
\end{example}

\begin{remark}
	Strictly speaking the use of Theorem~\ref{step two} in the examples just presented was not necessary. Whenever the base space $X$ is a projective space, or more generally a Fano complete intersection in a toric variety or flag bundle, then one can replace our use of Theorem~\ref{step two} (but not Theorem~\ref{step one}) by~\cite[Corollary~6.3.1]{CFKS2008}. However there are many examples that genuinely require both Theorem~\ref{step one} and Theorem~\ref{step two}: for instance when $X$ is a toric complete intersection but the line bundles that define the center of the blow-up do not arise by restriction from line bundles on the ambient space. (For a specific such example one could take $X$ to be the three-dimensional Fano manifold $\mathrm{MM}_{3\text{--}9}$: see~\cite[\S62]{CCGK16}.) For notational simplicity we chose to present examples with $X = \PP^N$, but the approach that we used applies without change to more general situations. 
\end{remark}
\bibliography{bibliography}
\bibliographystyle{alpha}

\end{document}